%% file: cb1.tex
\documentclass[11pt, reqno, a4paper,oneside]{amsart}
\usepackage{graphicx, epsfig, psfrag}
\usepackage{amsfonts,amsmath,amssymb,amsbsy,amsthm}
\usepackage{bm}
\usepackage{color}
\usepackage{float}
\usepackage{hyperref}
\usepackage{mathrsfs}
\usepackage{subfigure}
\usepackage{bbm}
\usepackage{longtable}

\usepackage{pdfsync}

\usepackage{environments}
\numberwithin{theorem}{section}
\numberwithin{equation}{section}

\def\vol{\rm{vol}}

\renewcommand{\paragraph}[1]{\subsubsection{#1}}

\input{notation}

\definecolor{cocol}{rgb}{0.7, 0, 0}
\definecolor{ftcol}{rgb}{0, 0, 0.9}

\begin{document}

\title{Nonlinear Elasticity from Atomistic Mechanics}

\author{C. Ortner}
\address{C. Ortner\\ Mathematics Institute \\ Zeeman Building \\
  University of Warwick \\ Coventry CV4 7AL \\ UK}
\email{christoph.ortner@warwick.ac.uk}



\author{F. Theil}
\address{F. Theil\\ Mathematics Institute\\ Zeeman Building \\
  University of Warwick \\ Coventry CV4 7AL \\ UK }
\email{f.theil@warwick.ac.uk}

\date{\today}

\thanks{This work was supported by the EPSRC Critical Mass Programme
  ``New Frontiers in the Mathematics of Solids'' (OxMoS), by the EPSRC
  grant EP/H003096 ``Analysis of atomistic-to-continuum coupling
  methods''.}



\begin{abstract}
  We present sharp convergence results for the Cauchy--Born
  approximation of general classical atomistic interactions, for both
  static and dynamic problems, for small data.
\end{abstract}

\maketitle


\section{Introduction}
The Cauchy--Born model (or, approximation) for {\em Bravais lattices}
is the most widely used nonlinear elasticity model of crystal
elasticity. It is obtained from the {\em Cauchy--Born rule}: the
stored energy per unit volume under a macroscopically homogeneous
deformation equals the energy per unit volume in the corresponding
homogeneous crystal. In some simple cases the Cauchy--Born
approximation for Bravais lattices can be justified as a statement
about minimum energy states \cite{friesecketheil02, cdkm06}. This
means that microstructural relaxation effects are ruled out. In less
restricted settings it has been shown to provide a highly accurate
approximation of crystal elastostatics \cite{E:2007a}.

It is highly desirable to obtain comparable results for evolutionary
problems because of the importance of phenomena linked to energy
transport and dissipation in crystallographic lattices
\cite{hltt08,ls08, ls11}.  The main contribution of this article is a
rigorous approximation error analysis of the Cauchy--Born wave
equation compared to the Newtonian equations of motion in the
atomistic model.  In this pursuit we are inspired by a recent effort
of Blanc, Le Bris and Lions \cite{BLBL:2012} who propose this problem
and solve it for various one-dimensional examples with different
classes of pair interactions.

Our own convergence result is formulated for a general class of
multi-body interactions in an infinite lattice, and only requires the
assumptions that the ``reference lattice'' is a ``stable'' Bravais
lattice (cf.~Section~\ref{sec:stab_small}) and that the interaction
strength decays sufficiently fast. We also provide a rigorous
approximation error analysis of the Cauchy--Born approximation for
static problems and the same conditions. The lattice stability
assumption is essential and cannot be removed (see
\S~\ref{sec:app_stab}).




Our main dynamic result, Theorem~\ref{th:w:mainthm}, is concerned with
the Newtonian equations of motion and the nonlinear Cauchy--Born wave
equation. In the scaling limit where both spatial and temporal
fluctuation are of the order $\eps$ (the lattice spacing) we prove
that the atomistic solution converges a solution of the Cauchy--Born
wave equation. Moreover, difference between the atomistic and the
Cauchy-Born solution is of order $\eps^2$. The result for the static
case (Theorem \ref{th:errsm:mainthm}) is analogous.

These strong results are the consequence of sharp quantitative links
between particle models and continuum models, which are also useful in
a wider context \cite{OrVK:bqce:2012, LuOr:acta, OrZha:patch3d}.  In
particular the localization technique
\eqref{eq:intro_localisation_trick} provides a continuous
representation of discrete objects. A key concept is the notion of
atomistic stress, which gives a natural weak form of the first
variation. {\em Pointwise} second-order consistency of the
Piola--Kirchhoff stress of the Cauchy--Born model with the atomistic
stress is established in Theorem~\ref{th:stress:c2}.  Error estimates
between local minimizers of the Cauchy--Born and atomistic model are
an immediate consequence of consistency. The proof in the dynamic case
is based on a similar pointwise consistency result for the divergence
of the stresses.

\subsection{Main results}
We assume that the {\em scaled} total atomistic energy $E^\eps$ can be
written as
\begin{align*}
E^\eps(u) = \eps^d \sum_{\xi \in \eps\L} V\B( \B\{
\smfrac{u(\xi+\eps\rho)-u(\xi)}{\eps} \B\}_{\rho\in\L\setminus\{0\}}\B),
\end{align*}
where $\L := \Z^d$, $u :\eps \L \to \R^d$ is a discrete displacement
and $V$ a multi-body potential describing the interaction between a
site $\xi$ and the rest of the body. The continuum limit is
characterized by the Cauchy--Born energy density function $W(\mF) =
V(\{\mF\rho\}_{\rho\in\L\setminus\{0\}})$ and the associated
functional $E(u) = \int_{\R^d} W(\nabla u)\, \dd x$. We will show that
$E$ characterizes the asymptotic behavior of solutions of static and
dynamic problems associated with $E^\eps$. We assume throughout this
section that $d \leq 3$ and that $\L$ is stable (see
\S~\ref{sec:stab_small}). \\[1em]
{\bf Theorem A (Elastostatics).}  {\em There exists constants $C_\mathrm{stat}, C,\delta_0
  >0$ such that for all $f:\R^d \to \R^d$, $f^\eps: \eps\L \to \R^d$
  satisfying
  \begin{align*}
    \|f\|_{1} \leq \delta_0 \quad \text{and} \quad
    \|f - f^\eps\|_{-1} \leq C_\mathrm{stat} \eps^2,
  \end{align*}
  and for $\eps$ sufficiently small, there exist local minimizers $u, u^\eps$ of the energies
  \begin{align*}
    E(u) - \int_{\R^d} f\cdot u \, \dd x \quad \text{and} \quad
    E^\eps(u) - \eps^d \sum_{\xi \in \eps\L} f^\eps(\xi) \cdot u^\eps(\xi),
  \end{align*}
  which also satisfy the bound $\| \nabla u - \nabla u^\eps\|_{L^2} \leq C \eps^2$.
}\\[1em]
The definition of the norms $\|\,\cdot\, \|_1$, $\|\,\cdot\,\|_{-1}$
and a precise formulation of the result are given in
Theorem~\ref{th:errsm:mainthm}.

\medskip \noindent {\bf Theorem B (Elastodynamics).} There exist
constants $\kappa, T, C > 0$ such that, for all initial states
$(u_0,u_1)$, $(u_0^\eps, u_1^\eps)$ satisfying the bounds
  \begin{align*}
    \| \nabla u_0 \|_{L^\infty} \leq \kappa \qquad \text{and} \qquad
    \| \nabla u_0 - \nabla u_0^\eps \|_{L^2} + \| u_1 - u_1^\eps
    \|_{L^2} \leq C_\mathrm{init}\, \eps^2,
  \end{align*}
  where $C_{\rm init}$ is independent of $\eps$, and for all
  sufficiently small $\eps$, there exist unique solutions $u \in
  C^2([0,T],H^2(\R^d))$ and $u^\eps \in C^2([0, T], \ell^2(\L))$ of
  the Cauchy problems
  \begin{align*}
    &\ddot{u} - {\rm div} \Scb(\nabla u)=0,
    & u(t=0)=u_0, \qquad \dot{u}(t=0) = u_1,\\
&\ddot{u}^\eps +\delta E^\eps(u^\eps) = 0,
& u^\eps(t=0) = u^\eps_0, \qquad \dot{u}^\eps(t=0) = u^\eps_1,
\end{align*}
where $\delta E^\eps$ denotes first variation of $E^\eps$, and
$\Scb_{ij}(F) = \frac{\partial W}{\partial \mF_{ij}}(\mF)$ is the
first Piola--Kirchhoff stress tensor in the Cauchy--Born model. The
solutions $u, u^\eps$ satisfy the estimate
\begin{equation}
  \label{eq:intro_errest_dyn}
  \max_{0 \leq t \leq T} \b( \|\nabla u(t) - \nabla \tilde{u}^\eps(t) \|_{L^2} + \| \dot{u} -
\dot{u}^\eps(t) \|_{L^2} \b) \leq C \eps^2.
\end{equation}

The definition of the norm $\|\,\cdot\,\|_3$ and the precise
formulation of the result is given in Theorem~\ref{th:w:mainthm}. We
note that the condition $\| \D u_0 \|_{L^\infty} \leq \kappa$ is a
fairly mild condition, which only ensure that $u_0$ is a ``stable
configuration''.

We make no statement about the maximal time interval for which the
estimate \eqref{eq:intro_errest_dyn} holds. Such a statement could be
made provided one establishes an atomistic G\aa rding
inequality. Moreover, such a result would also allow us to treat large
deformations in the static case. We stress, however, that our
estimates hold for a {\em macroscopic time interval}.

\subsection{New Concepts}
The comparison of discrete displacements $u^\eps$ with continuous
displacements $u$ is achieved by the introduction of several
approximation operators derived from a nodal basis function $\zeta$
with compact support (see \S~\ref{sec:interp} for the details).  We
define the Lipschitz continuous interpolation $u^\eps$ and a
quasi-interpolation $\tilde u^\eps$, which has a Lipschitz-continuous
gradient. The key property property of the framework is that it
delivers an integral representation of finite differences
\begin{align}
 \label{eq:intro_localisation_trick}
 \frac{\tilv(\xi+\eps\rho) - \tilv(\xi)}{\eps} = \int_\Om \ww_{\xi,\rho}^\eps(x) \D_\rho \barv(x) \dx,
  \quad \text{where} \quad
  \ww_{\xi,\rho}^\eps(x) := \eps^{-1} \int_0^1 \zz(\xi+t\eps\rho-x)\,\dt.
\end{align}
The kernel $\ww_{\xi,\rho}^\eps$ can be understood as a mollified
version of the line measure supported on the bond $\{ \xi+t\rho \sep t
\in [0,\eps]\}$.

A version of this technique was proposed by Shapeev
\cite{Shapeev:localisation} for the construction of
atomistic-to-continuum coupling schemes. Here we propose an
analytically convenient and stable variant of his idea, which is
analyzed in detail in \cite{OrShSu:2012}. 

Using \eqref{eq:intro_localisation_trick} we can construct an integral
represenation of the first variation of the atomistic energy,
\begin{equation}
  \label{eq:intro_delEeps}
  \b\< \del\E^\eps(u^\eps), \tilv \b\> =\, \int_\Om \mS^\eps(u^\eps; x) : \D v(x)
  \dx,
\end{equation}
for all compactly supported virtual displacements $v$, where
$\mS^\eps$ is the {\em stress} associated with a discrete displacement
$u^\eps$,
\begin{align} \label{eq:intro_stressdef} \mS^\eps(u^\eps; x) =\,& \sum_{\xi \in
    \eps\L} \sum_{\xi' \in \eps\L \setminus \{\xi\}} \b[f_{\xi,\xi'}
  \otimes (\xi'-\xi)\b] \ww_{\xi,\rho}^\eps(x);
\end{align}
here, $f_{\xi,\xi'}$ is the force acting between atoms $\xi$ and
$\xi'$ due to the site-energy associated with atom $\xi$. A precise
definition of the atomistic stress is given in \S~\ref{th:Sa}. We will
prove in \S~\ref{sec:stress:2ndorder} that $|\mS^\c - \mS^\eps| =
O(\eps^2)$ and in \S~\ref{sec:w:cons} that $|{\rm div} (\mS^\c -
\mS^\eps)| = O(\eps^2)$, which are the key technical ingredient
required for proving Theorems A and B.

\begin{remark}
  The tensor $\mS^\eps$ is closely related to the stress defined by
  Hardy \cite{Hardy}. Indeed, equation (4.5) in that work is
  essentially an Eulerian version of \eqref{eq:intro_stressdef}, for
  pair interactions, with a generic weighting function
  $\ww_{\xi,\rho}^\eps$. A general account of stress in molecular dynamics
  simulations is given in the recent work of Admal and Tadmor
  \cite{AdTa:2010}.

  A concern discussed in \cite{AdTa:2010} is the {\em non-uniqueness}
  of stress: note that $\mS^\eps$ is defined by
  \eqref{eq:intro_delEeps} only up to a divergence-free
  tensor. Indeed, it is possible to add ``discrete null-Lagrangians''
  to the atomistic energy, or decompose the atomistic energy in
  different ways, which would lead to different definitions of the
  atomistic stress $\mS^\eps$.

  Concerning this question, our work provides a selection mechanism
  based on comparison with the Cauchy--Born stress. If $V$ is selected
  to be ``as local as possible'' (cf. \S~\ref{sec:intro:V:decya}) and
  to satisfy an inversion symmetry
  (cf. \S~\ref{sec:intro:V:symmetry}), then the atomistic stress
  $\mS^\eps$ and the Cauchy--Born stress $\mS^\c$ are second-order
  close. Since there is a natural definition for $\mS^\c$, our version
  of the atomistic stress is reasonable whenever the atomistic
  configuration is ``close'' to an elastic continuum
  configuration.
\end{remark}

\subsection{Outline}
Since we work in an infinite domain, and admit an infinite interaction
range, even the definition of the atomistic energy is
non-trivial. Section \ref{sec:prelims} is devoted to this task. At the
same time we establish various auxiliary results that are useful for
the subsequent analysis. In Section \ref{sec:intro:cb} we define and
analyze the Cauchy--Born approximation; in particular, we establish
differentiability of the stored energy function and establish a
convenient functional analytic setting. In Section \ref{sec:stress} we
derive and analyze the atomistic stress, which plays a prominent role
in our analysis.  In Section \ref{sec:errsm} we present a rigorous
approximation error analysis of the static Cauchy--Born approximation.
Finally, in Section \ref{sec:wave} we establish approximation error
estimates between the solutions of Newton's equations of motion and
the Cauchy--Born wave equation.

\subsection{Summary of notation}
\label{sec:intro:notation}
Throughout, $\R$ denotes the real numbers, and $d \in \N$ the space
dimension.

Positions in space are usually denoted by $x, y, z \in \R^d$, while
lattice sites are denoted by $\xi,\eta \in \L := \Z^d$. Displacements,
either continuous or discrete, are denoted by $u, v, w$. Lattice
directions are denoted by $\rho,\vsig,\tau \in \Rg := \Z^d
\setminus\{0\}$.

Matrices are denoted by capital letters, $\mA, \mF, \mG, \mS \in \R^{d
  \times d}$. In particular, $\mA$ is reserved for the lattice
orientation (beginning of \S~\ref{sec:prelims}), and $\mS$ for stress
tensors.

If $f : \R^d \to \R^m$ is differentiable in $x$ then we denote its
Jacobi matrix by $\D f(x)$ and a direction derivative by $\D_r f(x) =
\D f(x) \cdot r$, $r \in \R^d$. Higher derivatives are denoted by
$\D^j f$ and are understood as $j$-linear forms with range in $\R^m$.

If $\mathscr{A}, \mathscr{B}$ are Banach spaces and $\mathscr{F} :
\mathscr{A} \to \mathscr{B}$ is $j$ times Fr\'echet differentiable (or
simply, differentiable) in $a$ then we denote its $j$-th derivative
(or, variation if $\mathscr{B} = \R$) by $\del^j \mathscr{F}$, which
is understood as a $j$-linear form with range in $\mathscr{B}$. If $a,
a_1, \dots, a_j \in \mathscr{A}$, then we write $\del^j
\mathscr{F}(a)[a_1, \dots, a_j]$ to evaluate this form.

If $\ell : \mathscr{A} \to \R$ is a linear functional, then we write
$\ell(a) = \< \ell, a\>$. If $\E : \mathscr{A} \to \R$ is
differentiable, then we write $\del\E(a)[a_1] = \< \del\E(a),
a_1\>$. If it is twice differentiable then we will also write
$\ddel\E(a)[a_1, a_2] = \< \ddel\E(a) a_1, a_2\>$.

We use $L^p, W^{1,p}$, for $p \in [1,\infty]$, to denote the standard
Sobolev and Lebesgue spaces, usually on the domain $\R^d$ (and
otherwise specified). We will also employ the so-called {\em
  homogeneous Sobolev spaces}, $\WWh^{j,p}$, which are defined in
\S~\ref{sec:homsob}.  Negative-norm (dual) spaces are denoted by
$W^{-1,p} = (W^{1,p'})^*$ and $\WWh^{-1,p} = (\WWh^{1,p'})^*$, where
$p' := p / (p-1) \in [1,\infty]$ is the dual Sobolev index.

The discrete ``Lebesgue space'' is denoted by $\ell^p$, usually with
domain $\Z^d$ and otherwise specified. Since most discrete versions of
$W^{1,p}$ would be equivalent to $\ell^p$, we will not define such a
space. However, we make heavy use of {\em discrete homogeneous Sobolev
  spaces} $\Ush^{1,p}$, which are defined in \S\ref{sec:interp}.

Finally, we make the convention that ``$\lesssim$'' stands for ``$\leq
C$'', where $C$ is a generic constant that may not depend on any data
in the model, nor on any functions involved in the inequality. In
particular, we will make explicit all dependence on the interaction
potential, which is crucial since we admit an infinite interaction
range.

\section{The Atomistic Energy}
\label{sec:prelims}
We formulate an atomistic model with classical multi-body interactions
on an infinite Bravais lattice. As discrete reference domain, we
choose $\L := \Z^d$, where $d \in \N$ (later restricted to $d \leq 3$)
is fixed throughout. We admit deformations of the form $y(\xi) = \mA
\xi + u(\xi)$, where $u$ is an unknown displacement and $\mA \in \R^{d
  \times d}$, $\det\mA > 0$, defines the reference state of the
system, $\mA \cdot \L$, which may be an arbitrary Bravais lattice.

We now present a formal definition of the atomistic potential energy,
which we make rigorous throughout the remainder of the section. Let
$\Rg := \L \setminus \{0\}$ denote the set of lattice directions.  For
discrete maps $v : \L \to \R^m$, $m \in \N$, we define the finite
differences and finite difference stencils
\begin{align*}
  \Da{\rho} v(\xi) :=\,& v(\xi+\rho) - v(\xi), \qquad \text{for } \xi
  \in \L, \rho \in
  \Rg, \quad \text{and} \\
  Dv(\xi) :=\,& \b( \Da{\rho} v(\xi) \b)_{\rho \in \Rg}, \qquad
  \text{for } \xi \in \L.
\end{align*}
A convenient space of finite-difference stencils is
\begin{displaymath}
  \DV := \b\{ \bfg = (g_\rho)_{\rho \in
  \Rg} \bsep g_\rho \in \R^d, \|\bfg\|_{\DV} < \infty \b\},
\end{displaymath}
equipped with the norm $\|\bfg\|_{\DV} := \max_{\rho \in \Rg}
|g_\rho| / |\rho|$.

Next, we assume that there exists a site-energy $V : \DV \to \R
\cup\{\pm\infty\}$, so that the atomistic potential energy of a
displacement $u : \L \to \R^d$ can be written, {\em formally}, as
\begin{equation}
  \label{eq:defn_Ea_pre}
  \Ea(u) := \sum_{\xi \in \L} \Es_\xi(u), \quad \text{where }  \Es_\xi(u)
  := V(Du(\xi)).
\end{equation}
Note that the site-energies are not well-defined for general $u$, and
moreover their sum need not exist.  In the remainder of this section,
we introduce a discrete function space setting in which we can make
\eqref{eq:defn_Ea_pre} rigorous.

We remark that $V$ implicitly depends on $\mA$, but since $\mA$ is
fixed throughout, we suppress this dependence.



\subsection{Interpolation of lattice functions}
\label{sec:interp}
We denote the set of all vector-valued lattice functions by $\Us$ and
those with compact support by $\Usz$:
\begin{displaymath}
  \Us := \b\{ v : \L \to \R^d \b\} \quad \text{and} \quad
  \Usz := \b\{ v \in \Us \bsep \supp(v) \text{ is compact} \b\}.
\end{displaymath}
To facilitate the transition between continuous and discrete maps we
introduce two \mbox{(quasi-)} interpolants of lattice functions.

Let $\zz \in \WW^{1,\infty}(\R^d; \R)$ satisfy $\zz(\xi) = 0$ for $\xi
\in \L \setminus\{0\}$ and $\zz(0) = 1$. We understand $\zz(\;\cdot\; -
\xi)$ as a nodal basis function associated with the site $\xi$, and
define the {\em first-order interpolant}
\begin{equation}
  \label{eq:interp:S1_interp}
  \barv(x) := \sum_{\xi \in \L} v(\xi) \zz(x - \xi), \qquad
  \text{for } v \in \Us.
\end{equation}
We shall assume throughout that $\zz \geq 0$, $\zz$ has compact
support, $\zz$ is symmetric about the origin $\zz(-x) = \zz(x)$, and
that the associated interpolation operator reproduces affine
functions: $\sum_{\xi\in \L} (a+b\cdot \xi) \zz(x-\xi) = a + b \cdot
x$ for all $a \in \R, b \in \R^d$. The latter property implies, in
particular, that $\int_{\R^d} \zz \dx = 1$.

Following \cite{Shapeev:localisation}, we also define a
quasi-interpolant obtained through convolution of $\barv$ with $\zz$:
\begin{align}
  \label{eq:interp:defn_tildev}
  \tilv(x) :=\,& (\zz \ast \barv)(x) = \int_{\R^d}
  \zz(x - y) \barv(y) \dy.
\end{align}
In general, $v(\xi) \neq \tilv(\xi)$ for $\xi \in \L$, hence it is
only a quasi-interpolant. The introduction of this second interpolant
leads to a set of techniques centered around the {\em localization
  formula} \cite{Shapeev:localisation, OrShSu:2012} (cf. the scaled
version in \eqref{eq:intro_localisation_trick})
\begin{equation}
  \label{eq:localisation_trick}
  D_\rho \tilv(\xi) = \int_{\R^d} \ww_{\xi,\rho}(x) \D_\rho v(x) \dx,
  \qquad \text{where } \quad
  \ww_{\xi,\rho}(x) := \int_0^1 \zz(\xi+t\rho-x) \dt,
\end{equation}
which yield surprisingly strong connections between the atomistic
model and its continuum counterpart (in particular, the definition and
analysis of the atomistic stress in \S\ref{sec:stress}). To prove
\eqref{eq:localisation_trick}, we simply note that
\begin{displaymath}
  D_\rho \tilv(\xi) = \int_0^1 \D_\rho
  \tilv(\xi+t\rho) \dt
  = \int_\Om \int_0^1 \zz(\xi+t\rho-x) \dt \, \D_\rho \barv(x)
  \dx.
\end{displaymath}

\begin{remark}
  A canonical choice for $\zz$ is the Q1-nodal basis function
  \begin{displaymath}
    \zz(x) := {\textstyle \prod_{i = 1}^d} \max\b(0, 1 - |x_i|\b),
  \end{displaymath}
  then ${\rm supp}(\zz) = [-1,1]^d$ and $\zz$ is piecewise
  multi-linear. In this case $\{ \barv \sep v \in \Us \}$ is the Q1
  finite element space, or equivalently, the space of tensor product
  linear B-splines and $\{ \tilv \sep v \in \Us \}$ is the space of
  cubic tensor product B-splines \cite{Hollig2003}.  However, other
  choices are equally possible, and indeed necessary in some
  situations \cite{OrVK:bqce:2012}. Since none of our results require
  explicit knowledge of the type of interpolant, we admit the most
  general case.
\end{remark}

\medskip We now collect several auxiliary results on the lattice
interpolants introduced above, all of which are established in
\cite{OrShSu:2012}.

\def\culo{c_1}
\def\cuhi{c_2}
\def\cDulo{c_0'}
\begin{lemma}
  \label{th:interp:basic_prop}
  \label{th:norm_equiv_Lp}
  Let $v \in \Us$, then its first-order interpolant
  \eqref{eq:interp:S1_interp} belongs to $\WW^{1,\infty}_\loc(\R^d;
  \R^d)$, and $\tilv \in \WW^{3,\infty}_\loc(\R^d; \R^d)$. Moreover,
  for any $p \in [1, \infty]$,
  \begin{align}
    \label{eq:interp:norm_equiv_Lp}
    \| v \|_{\ell^p} \lesssim \|\tilv \|_{\ell^p} \leq \| \tilv \|_{\LL^p}
    \leq \| \barv \|_{\LL^p} \lesssim \| v \|_{\ell^p}&
    \qquad \forall v \in \Us, \quad \text{and} \\
    \label{eq:interp_stab}
    \| \D \barv \|_{\LL^p}
    \lesssim \| \D \tilu \|_{\LL^p}
    \leq \| \D \barv \|_{\LL^p}&
    \qquad \forall v \in \Us.
  \end{align}
  (All constants in the above estimates are independent of $p$.)
\end{lemma}

\medskip
Next, we state a useful embedding result; of particular interest is
the case $j = m = 1$ and $q = \infty$, which states (employing
\eqref{eq:interp_stab}) that $\| \D u \|_{L^\infty} \lesssim \| \D u
\|_{L^p}$ for all $p \in [1, \infty]$. The proof uses the fact that
the ``mesh size'' in $\Z^d$ is one, and that $\ell^p \subset \ell^\infty$.

\begin{lemma}
  \label{th:embedding}
  Let $0 \leq j \leq m \leq 3$ and $p \leq q \in [1, \infty]$, then
  $\| \D^m \tilu \|_{\LL^q} \lesssim \|\D^j \tilu \|_{\LL^p}$ for all
  $u \in \Us$.
\end{lemma}



\subsection{The space of admissible displacements}
\label{sec:a:defm_space}
Since the atomistic model is formally translation invariant, we define
equivalence classes
\begin{displaymath}
  [u] := \{ u + t \sep t \in \R^d\}, \qquad \text{for } u \in \Us,
\end{displaymath}
and, for $p \in [1, \infty]$, define corresponding function spaces
\begin{equation}
  \label{eq:defn_Us12}
  \Ush^{1,p} := \b\{ [u] \bsep u \in \Us, \| \D u \|_{\LL^p} <
  +\infty \b\}.
\end{equation}
We will not make the distinction between $u$ and $[u]$, whenever it is
possible to do so without confusion, for example, when a statement or
function with argument $u$ is translation invariant.

\begin{proposition}
  \label{th:interp:energy-space}
  $\Ush^{1,p}$ is a Banach space. For $p \in (1, \infty)$, the
  subspace $\{ [v] \sep v \in \Usz\}$ is dense in $\Ush^{1,p}$. 
\end{proposition}


\medskip \noindent
Let $u \in \Ush^{1,p}$, $p < \infty$, be a discrete displacement and
$y(\xi) := \mA \xi + u(\xi)$ be the associated deformation. It is
shown in \cite{OrShSu:2012} that $|u(\xi)| \ll |\xi|$ as $|\xi| \to
\infty$. Therefore, the discrete deformation satisfies the far-field
boundary condition
\begin{displaymath}
  y(\xi) = \mA \xi + o(|\xi|) \quad \text{as } |\xi| \to \infty.
\end{displaymath}

Aside from satisfying a boundary condition we also require that
deformations are injective. Since our main results do not cover
arbitrarily large deformations, we will circumvent the question of
injectivity by placing an $L^\infty$-bound on the displacement
gradient (this will be ensured through conditions on the external
forces). Notationally, we fix a constant $\inv > 0$, and define
\begin{equation}
  \label{eq:defn_Ba}
  \Ba := \b\{ u \in \Us \bsep |D_\rho u(\xi)| \leq \inv \text{ for all
    } \xi \in \L, \rho \in \Rg \b\}.
\end{equation}

If $\inv$ is chosen sufficiently small, then displacements belonging
to $\Ba$ give rise to injective deformations. Indeed, if $\kappa < \|
\mA^{-1} \|^{-1}$, where $\|\; \cdot \; \|$ denotes the $\ell^2$-operator norm, then
\begin{equation}
  \label{eq:injective}
  \b| y(\xi) - y(\eta)\b| \geq \mu \b|\mA(\xi-\eta)\b| \qquad \forall \xi, \eta
  \in \L,
\end{equation}
where $\mu :=  1 - \kappa \|\mA^{-1}\| > 0$.


To conclude the discussion of discrete function spaces, we note that
we can define a smooth nodal interpolant on $\Ush^{1,p}$, which will
be useful in interpreting our results.

\begin{lemma}
  \label{th:interp:Itil}
  Let $u \in \Ush^{1,p}$, $p \in [1,\infty]$, then there exists $w \in
  \Ush^{1,p}$ such that $u = \tilw =: \Itil u$. Moreover,
  \begin{equation}
    \label{eq:D_stab_Itil}
    \| \D \baru \|_{L^p} \lesssim \| \D \Itil u \|_{L^p} \lesssim \|
    \D \baru \|_{L^p} \qquad \forall u \in \Ush^{1,p}.
  \end{equation}
\end{lemma}

\subsection{Assumptions on the interaction potential}
\label{sec:intro:V}
\subsubsection{Energy difference}
\label{sec:intro:V:difference}
We assume throughout that $V(\bfO) = 0$. Physically, this condition
means that $\Ea$ is an energy difference between the deformed state
$y(\xi) = \mA\xi + u(\xi)$ and the reference state $y(\xi) = \mA \xi$,
which may have infinite energy.

\subsubsection{Smoothness}
\label{sec:intro:V:smoothness}
We assume that $V$ is ``smooth'' at injective configurations; more
precisely, we define $\DV_\inv := \{ \bfg \in \DV \sep \| \bfg
\|_{\DV} \leq \inv \}$, and assume that $V \in \CC^k(\DV_\inv)$, for
some $k \geq 2$ and for all $\kappa$ that are sufficiently small.

For $\bfg \in \DV_\inv$ and for any ``multi-index'' $\bfrho \in
\Rg^j, 1 \leq j \leq k$, the partial derivative
\begin{align*}
  V_\bfrho(\bfg) := \frac{\pp^j V(\bfg)}{\pp g_{\rho_1} \dots \pp
    g_{\rho_j}} \in \R^{d^j}
\end{align*}
exists; $V_\bfrho(\bfg)$ is understood as a multilinear form
acting on families of vectors $\bfh = (h_1, \dots, h_j)$;
$V_\bfrho(\bfg)[\bfh] = V_{\rho_1\cdots\rho_j}(\bfg)[h_1, \dots,
h_j]$.

We also define the associated partial derivatives of the site-energies
by $\Es_{\xi,\bfrho}(u) := V_\bfrho(Du(\xi))$, or,
\begin{displaymath}
  \Es_{\xi,\bfrho}(u)[D_\bfrho\bfv(\xi)] = \Es_{\xi,\rho_1\cdots\rho_j}[D_{\rho_1}
  v_1(\xi), \dots, D_{\rho_j} v_j(\xi)] = V_\bfrho(Du(\xi))[D_\bfrho\bfv(\xi)],
\end{displaymath}
where $D_\bfrho\bfv(\xi) = (D_{\rho_1} v_1(\xi), \dots, D_{\rho_j}
v_j(\xi))$ for $\rho_i \in \Rg, v_i \in \Ush$, $i=1, \ldots, j$.

\subsubsection{Symmetry}
\label{sec:intro:V:symmetry}
We assume throughout that $V$ satisfies the following inversion symmetry:
\begin{equation}
  \label{eq:intro:ptsymm}
  V\b( (-g_{-\rho})_{\rho \in \Rg} \b) = V(\bfg) \qquad \forall \bfg
  \in \DV.
\end{equation}
This condition is physically motivated by the fact that permutations
of atoms and isometries should not change the energy of a system.  The
requirement \eqref{eq:intro:ptsymm} then assumes that the global
energy was partitioned in a way that preserves this symmetry.

\begin{lemma}
  Let $\mF \in \R^{d \times d}$, $|\mF|\leq\inv$, then
  \begin{equation}
    \label{eq:intro:ptsymm_DV}
    V_{-\bfrho}(\mF\cdot\Rg) = (-1)^j V_{\bfrho}(\mF\cdot\Rg) \qquad
    \forall \bfrho \in \Rg^j, \quad 1 \leq j \leq k.
  \end{equation}
\end{lemma}
\begin{proof}
  Let $\bfg_\mF := \mF \cdot \Rg$, then $\bfg_\mF' := (-
  \mF(-\rho))_{\rho\in\Rg} = \bfg_\mF$. Since $V$ is $k$ times
  differentiable in $\DV_\inv$, we can differentiate
  \eqref{eq:intro:ptsymm} at $\bfg_\mF$. Evaluating
  the first and second derivatives gives the stated result.
\end{proof}

\subsubsection{Decay Hypothesis}
\label{sec:intro:V:decya}
Sufficiently rapid decay of the interatomic interaction is a crucial
ingredient in our analysis. We define the following basic bounds on
the interaction potential:
With this notation, we define the bounds
\begin{align*}
  m(\bfrho) :=\,& {\textstyle \prod_{i=1}^j |\rho_i|} \sup_{\bfg \in \DV_\inv} \|
  V_\bfrho(\bfg) \| \qquad
  \text{for } \bfrho \in \Rg^j, \quad 1 \leq j \leq k,
\end{align*}
where $\| \ell \| := \sup_{\bfh \in (\R^{d})^j, |h_1| = \dots = |h_j|
  = 1} \ell[h_1, \dots, h_j]$ for a $j$-linear form $\ell$. We shall
assume throughout that $V$ and $\kappa$ are such, that
\begin{equation}
  \label{eq:basic_decay}
  M^{(j)} := \sum_{\bfrho \in \Rg^j} m(\bfrho) < +\infty, \qquad \text{for } 1 \leq j \leq k,
\end{equation}
which will ensure that $\Ea$ is $k$ times Fr\'echet differentiable
(cf. Theorem \ref{th:defn_Ea}).

In order to describe the class of admissible potentials, we also
discuss the decay assumption we will require in the static and dynamic
approximation error analysis. To simplify the following expressions,
we define $|\bfrho| := \sum_{i = 1}^j |\rho_i|$, for $\bfrho \in
\Rg^j$, $j \in \N$.

\def\Ms{M_{\rm s}}
\def\Md{M_{\rm d}}
\def\ms{m_{\rm s}}
\def\md{m_{\rm d}}

Let $p \in [1,\infty]$ and $2 \leq j \leq k$. In the static analysis,
we will assume finiteness of certain
\begin{align*}
  \Ms^{(j,p)} :=\,& \sum_{\bfrho\in\Rg^j} \ms^{(p)}(\bfrho), \quad \text{where}
  \\
  \ms^{(p)}(\bfrho) :=\,& m(\bfrho) |\bfrho|^2 \B( {\textstyle \sum_{i =
      2}^j \b(|\rho_1 \times \rho_i| + |\rho_1| + |\rho_i|\b) }\B)^{\smfrac{1}{(j-1)p}},
    \quad \text{for } \bfrho \in \Rg^j.
\end{align*}
In the dynamic analysis we will assume finiteness of certain constants
of the form
\begin{align*}
  \Md^{(j,p)} :=\,& \sum_{\bfrho\in\Rg^j} \md^{(p)}(\bfrho), \quad \text{where}
  \\
  \md^{(p)}(\bfrho) :=\,& \frac{m(\bfrho) |\bfrho|^3}{|\rho_1|} \B( {\textstyle \sum_{i =
      2}^j \b(|\rho_1 \times \rho_i| + |\rho_1| + |\rho_i|\b)} \B)^{\smfrac{1}{(j-1)p}},
    \quad \text{for } \bfrho \in \Rg^j.
\end{align*}
These constants naturally arise in the modeling error estimates
established in \S\ref{sec:stress:moderr} and \S\ref{sec:w:cons}.

We stress that, while finiteness of $M^{(j)}$, $1\leq j\leq k$, is a
standing assumption, we will assume finiteness of $\Ms^{(j,p)}$ and
$\Md^{(j,p)}$, for certain choices of $j$ and $p$, only when required.

\begin{remark}
  If $\Ea$ contains only pair interactions, then we can write $V$ in
  the form
  \begin{displaymath}
    V(\bfg) = \frac12 \sum_{\rho\in\Rg}\b[ \varphi(|g_\rho|) -
    \varphi(|\mA\rho|) \b],
  \end{displaymath}
  which clearly satisfies the symmetry \eqref{eq:intro:ptsymm}.
  Moreover, we show in \S\ref{sec:ex_lj} that, if $\varphi^{(j)}(r)
  \lesssim r^{-\alpha-j}$ for $r \geq 1$, $1 \leq j \leq k$, then
  \begin{displaymath}
    M^{(j)} \lesssim \sum_{\rho\in\Rg} |\rho|^{-\alpha}, \quad
    \text{and} \quad
    \Ms^{(j,2)} + \Md^{(j,2)} \lesssim \sum_{\rho\in\Rg} |\rho|^{5/2-\alpha}.
  \end{displaymath}
  It therefore follows that the constants $M^{(j)}$ are finite
  provided that $\alpha > d$ and that $\Ms^{(j,2)}, \Md^{(j,2)}$ are
  finite provided that $\alpha > d + 5/2$. In particular, this implies
  that the Lennard-Jones potential, $\varphi(r) = r^{-12} - 2 r^{-6}$,
  is included in our analysis.

  In \S\ref{sec:examples} we discuss other commonly employed
  potentials and show that they fall within our assumptions.
\end{remark}

\subsection{Definition of the atomistic energy}
\label{sec:defn_Ea}
We mentioned at the beginning of \S\ref{sec:prelims} that, due to the
infinite domain and the infinite interaction range, the definition of
the energy \eqref{eq:defn_Ea_pre} is non-trivial. The purpose of this
section is to give \eqref{eq:defn_Ea_pre} a rigorous interpretation.

\begin{theorem}[Regularity of $\Ea$]
  \label{th:defn_Ea}
  (i) If $u \in \Usz \cap \Ba$, then $(\Es_\xi(u))_{\xi\in\L}
  \in \ell^1$; that is, $\Ea(u)$ given by \eqref{eq:defn_Ea_pre} is
  well-defined.

  (ii) $\Ea : (\Usz \cap \Ba, \|\;\cdot\;\|_{\Ush^{1,2}}) \to \R$ is
  continuous; that is, there exists a unique continuous extension to
  $\Ush^{1,2} \cap \Ba$, which we still denote by $\Ea$.

  (iii) $\Ea \in C^k(\Ush^{1,2} \cap \Ba)$, with
  \begin{align}
    \label{eq:djEa}
    \del^j\Ea(u)[\bfv] =\,& \sum_{\xi\in\L}
    \sum_{\bfrho \in \Rg^j}
    \Es_{\xi,\bfrho}(u)[D_\bfrho \bfv(\xi)]
  \end{align}
  for $j = 1, \dots, k$, for all $u \in \Ush^{1,2}$ and $\bfv \in
  \Usz^j$.
  %
  %
  Moreover, if $1 \leq j \leq k$ and
  $\sum_{i = 1}^j \smfrac{1}{p_i} = 1$, $p_i \in [1, \infty]$, then
  \begin{equation}
    \label{eq:bound_delj}
    \b|\del^j\Ea(u)[\bfv]\b| \lesssim M^{(j)} \prod_{i = 1}^j \|
    \D \barv_i \|_{L^{p_i}}.
  \end{equation}
\end{theorem}

\medskip \noindent The proof of Theorem \ref{th:defn_Ea} is
established throughout the remainder of the section. We begin by
establishing a simple bound on finite differences, which gives a first
glimpse of the localisation technique used at crucial steps throughout
the paper.

\begin{lemma}
  \label{th:discrete_H1_norm}
  Let $\rho \in \Rg$ and $p \in [1, \infty]$, then
  \begin{align*}
    \| D_\rho \tilu \|_{\ell^p} \leq \| \D_\rho \baru \|_{L^p} \leq
    |\rho| \| \D \baru \|_{L^p} \qquad \forall u \in \Ush^{1,p}.
 \end{align*}
\end{lemma}
\begin{proof}
  The case $p = \infty$ is trivial, hence suppose that $p <
  \infty$. Using the localisation formula
  \eqref{eq:localisation_trick}, the fact that $\{ \zz(\;\cdot\;-\xi)
  \sep \xi \in \L \}$ is a partition of unity (and hence $\int_\Om
  \zz(x-\xi) \dx = 1$), we can estimate
  \begin{align*}
    \sum_{\xi\in\L} \b|\Da{\rho}\tilu(\xi)\b|^p
    =\,& \sum_{\xi}\B| \int_{t = 0}^1 \int_\Om
    \zz(\xi+t\rho - x) \Dc{\rho} \baru(x) \dx\dt \B|^p \\
    \leq\,& \sum_{\xi} \int_{t = 0}^1 \int_\Om
    \zz(\xi+t\rho - x) \b| \Dc{\rho}\baru(x)\b|^p \dx\dt \\
    =\,&  \int_\Om
    \b| \Dc{\rho}\baru(x)\b|^p \dx \leq |\rho|^p \| \D
    \baru\|_{\LL^2}^p. \qedhere
 \end{align*}
\end{proof}





\begin{lemma}
  \label{th:variation_bounds}
  Let $\bfg_\xi \in \DV_\inv$, $\xi \in \L$, and $v_i \in \Ush^{1,p_i}, 1
  \leq i \leq j \leq k$, with $\sum_{i = 1}^j \smfrac{1}{p_i} = 1$,
  then \\[-5mm]
  \begin{displaymath}
    \sum_{\xi \in \L} \sum_{\bfrho\in\Rg^j} \b|V_{\bfrho}(\bfg_\xi)
    [D_\bfrho \bfv(\xi)]\b| \lesssim M^{(j)} \prod_{i=1}^j \| \D \barv_i \|_{L^{p_i}}.
  \end{displaymath}
\end{lemma}
\begin{proof}
  Let $w_i \in \Ush^{1,p_i}$ such that $\tilw_i = \Itil v_i$, then
  \begin{displaymath}
     \b|V_{\bfrho}(\bfg_\xi)
    [D_\bfrho \bfv(\xi)] \b| \leq \smfrac{m(\bfrho)}{\prod_{i=1}^j|\rho_i|} \prod_{i = 1}^j
    |D_{\rho_i} \tilw_i(\xi)|.
  \end{displaymath}
  Summing over $\xi \in \L, \bfrho \in \Rg^j$, applying the
  generalized H\"{o}lder inequality, followed by
  Lemma~\ref{th:discrete_H1_norm}, yields
  \begin{align*}
    \sum_{\bfrho \in \Rg^j} \sum_{\xi\in\L}\b|V_{\bfrho}(\bfg_\xi)
    [D_\bfrho \bfv(\xi)] \b| \leq\,& \sum_{\bfrho \in \Rg^j}
   \smfrac{m(\bfrho)}{\prod_{i=1}^j|\rho_i|} \sum_{\xi\in\L} \prod_{i = 1}^j
    |D_{\rho_i} \tilw_i(\xi)| \\
    \leq\,& \sum_{\bfrho \in \Rg^j}
   \smfrac{m(\bfrho)}{\prod_{i=1}^j|\rho_i|} \prod_{i = 1}^j \| D_\rho \tilw_i \|_{\ell^{p_i}}
    \leq \sum_{\bfrho \in \Rg^j} m(\bfrho) \|
    \D\barw_i \|_{L^p}.
  \end{align*}
  Applying \eqref{eq:D_stab_Itil} yields the stated result.
\end{proof}

\begin{remark}
  Lemma \ref{th:variation_bounds} implies, in particular, that we can
  interchange the summation in series of the form $\sum_{\xi\in\L}
  \sum_{\bfrho \in \Rg^j} V_\bfrho(\bfg_\xi)[D_\bfrho \bfv(\xi)]$,
  provided that $\bfg_\xi \in \DV_\inv$ and $\D v_i$ have sufficient
  integrability. We will henceforth perform interchanges of summations
  without further comment.
\end{remark}

\begin{lemma}
  \label{th:delEa0.eq.0}
  If $u \in \Usz$, then
  \begin{equation}
    \label{eq:delEa0.eq.0}
    \sum_{\xi\in\L}\sum_{\rho\in\Rg} \Es_{\xi,\rho}(0) \cdot D_\rho
    u(\xi) = 0.
  \end{equation}
\end{lemma}
\begin{proof}
  According to Lemma~\ref{th:variation_bounds}, the sum on the
  right-hand side of \eqref{eq:delEa0.eq.0} is well-defined, the
  summand belonging to $\ell^1$.

  Since $\Es_{\xi,\rho}(0) = V_{\rho}(\bfO)$, and interchanging the order of
  summation, we obtain
  \begin{displaymath}
    \sum_{\xi\in\L}\sum_{\rho\in\Rg} \Es_{\xi,\rho}(\bfO) \cdot D_\rho
    u(\xi) = \sum_{\rho \in \Rg} V_{\rho}(\bfO) \sum_{\xi\in\L}
    D_\rho u(\xi).
  \end{displaymath}
  For $u \in \Usz$ the right-hand side clearly vanishes.
\end{proof}

Motivated by the previous lemma, we define
\begin{equation}
  \label{eq:defn_Ea_alt}
  \hat{\Ea}(u) := \sum_{\xi \in \L} \hat{\Es}_\xi(u), \qquad \text{where
  } \hat{\Es}_\xi(u) := \Es_\xi(u) - \sum_{\rho\in\Rg}
  \Es_{\xi,\rho}(0)\cdot D_\rho u(\xi),
\end{equation}
then we immediately obtain the following result.

\begin{lemma}
  \label{th:hatEs}
  $\hat{\Es}_\xi \in C^k(\Ush^{1,p}\cap \Ba)$, for all $p \in [1,
  \infty]$, with
  \begin{align*}
    \hat{\Es}_{\xi,\rho}(u) =\,& \Es_{\xi,\rho}(u) - \Es_{\xi,\rho}(0), \quad
    \text{ for } \rho \in \Rg, \quad  \text{and} \\
    \hat{\Es}_{\xi,\bfrho}(u) =\,& \Es_{\xi,\bfrho}(u), \quad \text{ for }
    \bfrho \in \Rg^j, 2 \leq j \leq k.
  \end{align*}
\end{lemma}
\begin{proof}
  This result follows immediately from the fact that $\Es_\xi \in
  C^k(\Ush^{1,p} \cap \Ba)$ and the estimate
  \begin{displaymath}
    \sum_{\rho\in\Rg}
    |\Es_{\xi,\rho}(0)|\, |D_\rho u(\xi)| \lesssim \sum_{\rho\in\Rg}
    |\rho| m(\rho) \| \D u \|_{L^p} = M^{(1)} \| \D u \|_{L^p}. \qedhere
  \end{displaymath}
\end{proof}

\begin{lemma}
  \label{th:Ea_alt}
  Let $u \in \Ush^{1,2} \cap \Ba$, then $\hat{\Es}_\xi(u) \in \ell^1$;
  that is, \eqref{eq:defn_Ea_alt} is well-defined for all $u \in
  \Ush^{1,2}\cap\Ba$.  Moreover, $\hat{\Ea} \in
  C^k(\Ush^{1,2}\cap\Ba)$ with variations given by \eqref{eq:djEa}
  (with $\Ea$ replaced by $\hat\Ea$).
\end{lemma}
\begin{proof}
  Fix $u \in \Ush^{1,2}\cap\Ba$, then $\theta Du(\xi) \in \DV_\inv$
  for all $\xi \in \L$ and $\theta \in [0, 1]$. Hence, we can expand
  \begin{equation}
    \label{eq:Ea_alt:20}
    \hat{\Es}_\xi(u) = \hat{\Es}_\xi(0) + \sum_{\rho\in\Rg}
    \hat{\Es}_{\xi,\rho}(0) \cdot D_\rho u(\xi) +
    \sum_{\rho,\vsig\in\Rg} V_{\rho\vsig}(\bfg_\xi) [D_\rho u(\xi),
    D_\vsig u(\xi)],
  \end{equation}
  where $\bfg_\xi = \theta_\xi Du(\xi) \in \DV_\inv$ for some $\theta_\xi
  \in (0, 1)$. By definition, $\hat{\Es}_\xi(0) = \hat{\Es}_{\xi,\rho}(0)
  = 0$, and hence,
  \begin{displaymath}
    \hat{\Es}_\xi(u) = \sum_{\rho,\vsig\in\Rg} V_{\rho\vsig}(\bfg_\xi) [D_\rho u(\xi),
    D\vsig u(\xi)] \qquad \forall \xi \in \L.
  \end{displaymath}
  It now follows from Lemma \ref{th:discrete_H1_norm} that
  $\hat{\Es}_\xi(u) \in \ell^1$, that is, $\hat{\Ea}(u)$ is
  well-defined.

  The differentiability of $\hat{\Es}$ can be shown using analogous
  arguments. 
\end{proof}

\begin{proof}[Proof of Theorem \ref{th:defn_Ea}]
  (i) If $u \in \Usz \cap\Ba$ then Lemma \ref{th:Ea_alt} implies
  that $\hat \Es_\xi(u) \in \ell^1$, while Lemma \ref{th:delEa0.eq.0}
  ensures that $(\hat \Es_\xi(u) - \Es_\xi(u)) \in \ell^1$.

  (ii) For $u \in \Usz \cap\Ba$, Lemma \ref{th:delEa0.eq.0}
  implies that $\Ea(u) = \hat\Ea(u)$. Since $\hat\Ea$ is continuous on
  $\Ush^{1,2}\cap\Ba$, and $\Usz$ is dense in $\Ush^{1,2}$, it follows
  that $\hat\Ea$ is the continuous extension of $\Ea$ to
  $\Ush^{1,2}\cap\Ba$.

  (iii) Since $\hat\Ea = \Ea$, this statement follows from Lemma
  \ref{th:Ea_alt}.  The simplified representation of $\del\Ea$
  (replacing $\hat\Phi_{\xi,\rho}$ with $\Phi_{\xi,\rho}$) follows
  from Lemma \ref{th:delEa0.eq.0}. For $j > 1$ and $\bfrho\in\Rg^j$ we
  have $\hat\Phi_{\xi,\bfrho} = \Phi_{\xi,\bfrho}$.
\end{proof}

\begin{remark}
  1. From the proof of Lemma \ref{th:Ea_alt} is becomes apparent why we
  assumed in \S\ref{sec:intro:V:smoothness} that $k \geq 2$. Only in
  that case are able to show that $\hat\Ea$ is well-defined on
  $\Ush^{1,2}\cap\Ba$. We note, however, that if we only assume $k = 1$,
  then the the definition \eqref{eq:defn_Ea_pre} still yields $\Ea \in
  C^1(\Ush^{1,1}\cap\Ba)$.

  2. To define $\Ea(u)$ it is not necessary to assume a bound on $\|
  \D u \|_{L^\infty}$. Indeed, if we require that $V$ is smooth at all
  configurations $Du(\xi)$ for displacements $u$ satisfying the
  injectivity requirement \eqref{eq:injective}, then Theorem
  \ref{th:defn_Ea} can be extended to this class of deformations.

  The key observation is that $|\rho|^{-1} |D_\rho u(\xi)| \to 0$
  uniformly in $\rho$ as $|\xi| \to \infty$, and also $|\rho|^{-1}
  |D_\rho u(\xi)| \to 0$ uniformly in $\xi$ as $|\rho| \to
  \infty$. (The first statement follows from the fact that $\xi
  \mapsto D_\rho u(\xi) \in \ell^2$; the second statement follows from
  the inequality $|\rho|^{-1}|D_\rho u(\xi)| \lesssim |\rho|^{-1/2} \|
  \D u \|_{L^2}$, which is easily established from
  \eqref{eq:localisation_trick}.) In particular, this implies that $D
  u(\xi) \in \DV_\inv$ for $|\xi|$ sufficiently large. One can now
  apply the expansion \eqref{eq:Ea_alt:20} in the far-field.
\end{remark}

\section{The Cauchy--Born Approximation}
\label{sec:intro:cb}
The Cauchy--Born elastic energy density function
$W : \R^{d \times d} \to \R\cup\{\pm\infty\}$ is defined by
\begin{align}
W(\mF) :=  V(\mF\cdot\Rg).
\end{align}
In the regime of ``smooth elastic'' deformations, the {\em
  Cauchy--Born} model is a popular approximation to the atomistic
model $\Ea$ \cite{BLBL:arma2002, E:2007a, Wallace98, Hudson:stab}.

Formally, if $u : \R^d \to \R^d$ is smooth, then $V(D u(\xi)) \approx
V( \D u(\xi) \cdot\Rg) = W(\D u(\xi))$ and hence we can approximate $\Ea(u)$
by
\begin{equation}
  \label{eq:defn_Ec}
  \Ec(u) := \int_{\Om} W(\D u) \dx.
\end{equation}
In the remainder of this section we introduce a function space setting
to make \eqref{eq:defn_Ec} rigorous, and establish associated
auxiliary results.


\subsection{Homogeneous Sobolev spaces}
\label{sec:homsob}
The formal Euler--Lagrange equation associated with $\Ec$ is a
second-order elliptic system in $\Om$. Due to the translation
invariance, convenient function spaces for equations of this type are
the {\em homogeneous Sobolev spaces} (or, {\em Beppo-Levi
  spaces}~\cite{DL}). Here, we use a variant of spaces of equivalence
classes:
\begin{align*}
  \WWh^{m,2} :=\,& \b\{ [u] \bsep u \in \WW^{m,2}_{\rm loc}(\R^d; \R^d),
  \D u \in W^{m-1, 2} \b\}, \qquad \text{ for } m = 1, 2, \dots.
\end{align*}
The space $\WWh^{m,2}$ is equipped with the norm
\begin{displaymath}
  \| u \|_{\WWh^{m,2}} := \b({\textstyle \sum_{1 \leq j \leq m}} \| \D^j u \|_{L^2}^2\b)^{1/2}.
\end{displaymath}
It is easy to see (see \cite{OrSu:homsob:2012} for a proof of the case
$m = 1$; the general case is analogous) that $\WWh^{m,2}$ is a Banach
space, and that the subspace $\{ [v] \sep v \in C^\infty_0 \}$ is
dense, where $C^\infty_0 := \{ u \in C^\infty(\R^d; \R^d) \sep {\rm
  supp}(u) \text{ compact}\}$.

The natural space of continuous displacements is $\WWh^{1,2}$. In
order to avoid non-interpenetration of matter we shall assume that all
displacement gradients satisfy a uniform bound. To that end we define
\begin{displaymath}
  \Bc := \b\{ u \in W^{1,\infty} \bsep \| \D u \|_{L^\infty} \leq \inv \b\},
\end{displaymath}
where $\inv$ is the same constant as in the definition of $\Ba$
\eqref{eq:defn_Ba}.

A straightforward extension of \cite[Thm. 2.2]{OrSu:homsob:2012} shows
that $|u(x)| \ll |x|$ as $|x| \to \infty$ for all $u \in \WWh^{1,2}
\cap \Bc$, hence the associated {\em deformation} $y(x) = \mA x +
u(x)$ again satisfies the far-field boundary condition $y(x) \sim \mA
x$ as $|x| \to \infty$.

\subsection{Definition of the Cauchy--Born energy}
\label{sec:c:defn_Ec}
In this section we make the definition of the Cauchy--Born energy
\eqref{eq:defn_Ec} rigorous. We first analyze the stored energy
function $W$ of the Cauchy--Born model.

\begin{lemma}
  \label{th:defn_W}
  $W \in C^k(\{ \mF \in \R^{d \times d} \sep |\mF| \leq \inv \})$. For
  $|\mF| \leq \inv$ and $(\mG_1, \dots, \mG_j) \in (\R^{d \times
    d})^j$, $1 \leq j \leq k$, we have
  \begin{displaymath}
    \delta^j W(\mF) [\mG_1, \dots, \mG_j] = \sum_{\bfrho \in \Rg^j} V_\bfrho( \mF
    \cdot \Rg) [\mG_1 \rho_1, \dots, \mG_j \rho_j].
  \end{displaymath}
\end{lemma}
\begin{proof}
  This result follows immediately from the fact that, for $|\mF| \leq
  \inv$, $\mF\cdot\Rg \in \DV_\kappa$.
\end{proof}

The derivative of $W$ can be represented by the first Piola--Kirchhoff
stress tensor,
\begin{equation}
  \label{eq:defn_Scb}
  \Scb_{i\alpha}(u; x) := \left.\frac{\pp W(\mF)}{\pp \mF_{i\alpha}}\right|_{\mF = \D u(x)} \qquad \text{for } i,\alpha = 1\ldots d, \; u \in \Bc.
\end{equation}
Then we can write $\del W(\D u)[\mG] = \Scb(u) : \mG$, where '$:$'
denotes the usual Frobenius inner product.

Following the arguments in \S\ref{sec:defn_Ea}, we obtain that $\Ec$
is well-defined in $C^\infty_0 \cap \Bc$ and has a continuous
extension to $\WWh^{1,2} \cap \Bc$. The proof is analogous to the
proof of Theorem \ref{th:defn_Ea}.

\begin{proposition}[Definition of $\Ec$]
  \label{th:defn_Ec}
  (i) If $u \in C^\infty_0 \cap \Bc$, then $W(\D u) \in L^1$; hence
  $\Ec(u)$ is well-defined by \eqref{eq:defn_Ec}.

  (ii) $\Ec : (C^\infty_0 \cap\Bc, \|\;\cdot\;\|_{\WWh^{1,2}}) \to \R$ is
  continuous; hence there exists a unique continuous extension to
  $\WWh^{1,2} \cap \Bc$, which we still denote by $\Ec$.

  (iii) $\Ec \in C^k(\WWh^{1,2}\cap \Bc;
  \|\;\cdot\;\|_{\WWh^{1,2}}+\|\;\cdot\;\|_{\WWh^{1,\infty}})$ with
  \begin{align}
    \label{eq:ddelEc}
    \del^j\Ec(u)[\bfv] =\,& \int_{\R^d} \del^j W(\D u)[\D v_1, \dots,
    \D v_j] \dx, \qquad  \text{for } \bfv \in (C^\infty_0)^j, \quad 1 \leq j
    \leq k.
  \end{align}
  Moreover, for $\sum_{i = 1}^j \smfrac{1}{p_i} = 1$, we have
  \begin{equation}
    \label{eq:bound_delj_Ec}
    \del^j\Ec(u)[\bfv] \leq M^{(j)} \prod_{i = 1}^j \| \D v_i \|_{L^{p_i}}.
  \end{equation}
\end{proposition}

\section{Stress}
\label{sec:stress}
For $u \in \WWh^{1,2} \cap \Bc$ the canonical representation of
$\del\Ec(u)$ is (cf. \eqref{eq:ddelEc} with $j = 1$)
\begin{equation}
  \label{eq:delEc}
  \b\< \del\Ec(u), v \b\> = \int_{\R^d} \Scb(u) : \D v \dx,
  \qquad \text{for } v \in C^\infty_0.
\end{equation}
The {\it first Piola--Kirchhoff stress} $\Scb$ is dual to the virtual
displacement gradient $\D v$.  By contrast, the first variation of the
atomistic energy is more commonly expressed in terms of forces, which
are dual to $v$. The purpose of this section is to derive and analyze
an atomistic concept of stress that yields a representation of
$\del\Ea$ analogous to \eqref{eq:delEc}. This will then be employed in
a sharp consistency analysis of the Cauchy--Born approximation in
\S~\ref{sec:stress:2ndorder}, \S~\ref{sec:stress:moderr}, and
\S~\ref{sec:w:cons}.

\subsection{An atomistic stress function}
The ``canonical weak form'' of $\del\Ea$, given in \eqref{eq:djEa}, is
\begin{equation}
  \label{eq:stress:delEa_v1}
  \b\<\del\Ea(u), v \b\> = \sum_{\xi\in\L}\sum_{\rho \in\Rg} \Es_{\xi,\rho}(u)
  \cdot D_\rho v(\xi), \qquad \text{for } v \in \Usz.
\end{equation}
To proceed, we fix $v \in \Usz$ but test $\del\Ea(u)$ with $\tilv$
instead of $v$. We apply the localization formula
\eqref{eq:localisation_trick} to obtain
\begin{align}
 \b\<\del\Ea(u), \tilv \b\>
 \notag
 =\,& \sum_{\xi\in\L}\sum_{\rho \in\Rg}  \Es_{\xi,\rho}(u) \cdot \int_\Om
 \ww_{\xi,\rho}(x) \, \Dc{\rho}\barv(x) \dx \\
 =\,& \int_\Om \bg\{ \sum_{\xi \in \L} \sum_{\rho \in \Rg}
 \b[\Es_{\xi,\rho}(u) \otimes \rho\b] \ww_{\xi,\rho}(x) \bg\} : \D\barv \dx.
  \label{eq:stress:delEa_v2}
\end{align}
Using the decay assumption \eqref{eq:basic_decay} one can apply
Fubini's theorem to justify the interchange of integral and sums, and
thus obtain the following result.

\begin{proposition}
  \label{th:Sa}
  Let $u \in \Ush^{1,2} \cap \Ba$, then
  \begin{equation}
    \label{eq:stress:defn_Sa}
    \begin{split}
      \b\< \del\Ea(u), \tilv \b\> =\,& \int_\Om \Sa(u;x) : \D \barv(x)
      \dx \qquad \forall v \in \Usz,\\
      \text{where} \quad \Sa(u; x) :=\,& \sum_{\xi \in \L} \sum_{\rho \in \Rg}
      \b[\Es_{\xi,\rho}(u) \otimes \rho\b] \ww_{\xi,\rho}(x),
    \end{split}
  \end{equation}
  and $\ww_{\xi,\rho}$ is defined in \eqref{eq:localisation_trick}.
\end{proposition}

\begin{remark}
  The fact that the test function in the left-hand side and
  right-hand side of~\eqref{eq:stress:defn_Sa} differ may seem
  counter-intuitive at first. Allowing this seeming discrepancy makes
  the rather simple definition of atomistic stress possible, and will
  lead to sharp consistency estimates requiring only modest analytical
  effort while the subsequent error analysis in \S\ref{sec:errsm} and
  \S\ref{sec:wave} requires only minor adjustments.
\end{remark}

\subsection{Second-order accuracy of the Cauchy--Born stress}
\label{sec:stress:2ndorder}
In this section we prove the following pointwise second-order
consistency estimate between the Cauchy--Born and atomistic stress
functions.

\begin{theorem}
  \label{th:stress:c2}
  Let $u \in \WWh^{3,\infty} \cap \Bc$, and $x \in \Om$, then
  \begin{align}
    \label{eq:stress:c2}
    \b| \Sa(u; x) - \Scb(u; x) \b| \lesssim\,&
    \sum_{\bfrho\in\Rg^2} \ms^{(\infty)}(\bfrho)\, \|
    \D^3 y \|_{L^\infty(x+\nu_{\rho_1,\rho_2})} \\
    \notag & + \sum_{\bfrho\in\Rg^3} \ms^{(\infty)}(\bfrho) \, \| \D^2
    y \|_{L^\infty(x+\nu_{\rho_1,\rho_2})} \| \D^2 y
    \|_{L^\infty(x+\nu_{\rho_1,\rho_3})},
  \end{align}
  where $\ms^{(\infty)}$ is defined in \S\ref{sec:intro:V:decya}, and
  for each $\rho,\vsig\in\Rg$ the set $\nu_{\rho\vsig}$ satisfies
  $\nu_{\rho\vsig} \subset B(0, c(|\rho|+|\vsig|))$, for some constant
  $c > 0$, $-\nu_{\rho\vsig}=\nu_{\rho\vsig}$, and
  \begin{equation}
    \label{eq:bounds_nu_rhosig}
    {\rm vol}(\nu_{\rho\vsig}) \lesssim |\rho\times \vsig| + |\rho| + |\vsig|.
  \end{equation}
\end{theorem}

Before we embark on the proof of Theorem \ref{th:stress:c2}, we
establish two useful identities for the weights $\ww_{\xi,\rho}$,
which will enable us to exploit the inversion symmetry
\eqref{eq:intro:ptsymm}.

\begin{lemma}
  \label{th:stress:phirho_Prop}
  Let $x, \rho \in \R^d$; then
  \begin{align}
    \label{eq:stress:phirho_Prop0}
    \sum_{\xi \in \L} \ww_{\xi,\rho}(x) =\,& 1, \qquad \text{and} \\
    \label{eq:stress:phirho_Prop1}
    \sum_{\xi \in \L} \ww_{\xi,\rho}(x) \,(\xi-x) =\,& -\smfrac12 \rho.
  \end{align}
\end{lemma}
\begin{proof}
  Both results rely on the assumption that affine functions are
  invariant under the first-order interpolant. Clearly, this is still
  true on a shifted grid: if $v : \R^d \to \R^d$ is affine, then for
  any $z, x \in \R^d$ we have
  \begin{equation}
    \label{eq:stress:phirhoprop:1}
    v(x) = \sum_{\eta \in (\L + z)} \zz(x - \eta) v(\eta).
  \end{equation}

  To prove \eqref{eq:stress:phirho_Prop0} we write out the left-hand
  side, and employ \eqref{eq:stress:phirhoprop:1} with $v(x) = 1$:
  \begin{align*}
     \sum_{\xi \in \L} \ww_{\xi,\rho}(x) =\,& \int_{s = 0}^1
     \sum_{\xi \in \L} \zz\b(\xi + (s\rho - x)\b) \ds = \int_{s=0}^1 1
     \ds = 1.
  \end{align*}

  To prove \eqref{eq:stress:phirho_Prop1}, let $s \in [0, 1]$ be
  fixed; then
  \begin{align*}
    \sum_{\xi \in \L} \zz\b((\xi-x)+s\rho\b) (\xi-x)
    = \sum_{\eta \in(x+ \L)} \zz(s\rho - \eta) (-\eta),
  \end{align*}
  were we substituted $\eta = - (\xi-x)$, and hence sum over $-(\L -
  x) = (x+\L)$. Employing again \eqref{eq:stress:phirhoprop:1} with
  $v(x') = -x'$, we obtain
  \begin{displaymath}
    \sum_{\xi \in \L} \zz\b((\xi-x)+s\rho\b) (\xi-x) = - s\rho,
  \end{displaymath}
  and integrating with respect to $s$ gives
  \begin{displaymath}
     \sum_{\xi \in \L} \ww_{\xi,\rho}(x) (\xi-x)
     = \int_{s = 0}^1 \sum_{\xi \in \L} \zz\b((\xi-x)+s\rho\b) (\xi-x)
     \ds
     = \int_{s = 0}^1 (-s\rho) \ds = - \smfrac12 \rho. \qedhere
  \end{displaymath}
\end{proof}

\begin{proof}[Proof of Theorem \ref{th:stress:c2}]
  %
  Throughout this proof, let $\nu_{\xi,\vsig} := \conv\{x,
  \xi,\xi+\vsig\}$, and
  \begin{displaymath}
    \eps_{\xi,\vsig} := \| \D^2 u \|_{\LL^\infty(\nu_{\xi,\vsig})} \quad
    \text{and} \quad
    \delta_{\xi,\vsig} := \| \D^3 u \|_{\LL^\infty(\nu_{\xi,\vsig})}.
  \end{displaymath}
  We will suppress all arguments where it is possible to do so without
  confusion, for example, $\Sa = \Sa(u; x)$ and $\Scb = \Scb(u; x)$.

  Defining the symbols
  \begin{displaymath}
    V_\bfrho := V_\bfrho(x) := V_\bfrho\b(\D u(x)\cdot\Rg\b),
    \quad \text{ for } \bfrho \in \Rg^j, \quad j = 1, 2.
  \end{displaymath}
  we can rewrite $\Scb = \sum_{\rho \in \Rg} V_\rho \otimes
  \rho$. In \eqref{eq:stress:phirho_Prop0} we have established that
  $\sum_{\xi\in \L}\ww_{\xi,\rho}(x) = 1$, which implies
  \begin{align}
    \Sa - \Scb
    \notag
    =\,& \sum_{\rho \in \Rg} \bg\{ \sum_{\xi \in \L} \b[\Es_{\xi,\rho} \otimes
    \rho\b] \ww_{\xi,\rho}(x) - \b[ V_{\rho} \otimes \rho \b]
    \bg\} \\
    \label{eq:stress:c2_5}
    =\,& \sum_{\rho \in \Rg} \sum_{\xi \in \L} \B\{ \b[\Es_{\xi,\rho} -
    V_{\rho}\b] \otimes
    \rho \B\} \ww_{\xi,\rho}(x).
  \end{align}
  Since $\zz$ has compact support, there exists a constant $c > 0$
  such that $\chi_{\xi,\rho}(x) = 0$ for all $\xi\in\L$ with $|\xi-x|
  > c |\rho|$. Hence, we will assume throughout the rest of the proof
  that $|\xi-x| \lesssim |\rho|$.

  We Taylor expand the term $[\Es_{\xi,\rho} - V_{\rho}]$ as
  follows:
  \begin{align}
    \label{eq:stress:c2_8}
    \Es_{\xi,\rho}  - V_{\rho}
    =\,& \sum_{\vsig \in \Rg} V_{\rho\vsig} [\;\cdot\;, \Da{\vsig} u(\xi) -
    \Dc{\vsig} u(x)] + E_1, \\
    \notag
    \text{where} \quad |E_1| \lesssim\,&  \sum_{\tau\vsig \in \Rg} (|\tau|+|\vsig|+|\rho|)^2    \smfrac{m(\rho, \vsig, \tau)}{|\rho|} \, \eps_{\xi,\vsig}\eps_{\xi,\tau}
  \end{align}
  The details of the estimates for the remainder $E_1$ are easily
  established, the key observation being that
  \begin{displaymath}
    |D_\vsig u(\xi) - \D_\vsig u(x)| \leq (\smfrac12
    |\vsig| + |\xi-x|) |\vsig| \eps_{\xi,\vsig} \lesssim (|\vsig| +
    |\rho|) |\vsig| \, \eps_{\xi,\vsig},
  \end{displaymath}
  which is obtained by expanding along the segments $\conv\{x, \xi\}$
  and $\conv\{\xi,\xi+\vsig\}$.  Here, and in the following, we skip
  the details required for estimating the remainders.

  Expanding $D_\vsig u(\xi) - \D_\vsig u(x)$,
  \begin{align}
    \notag
    \Da{\vsig} u(\xi) - \Dc{\vsig} u(x)
    =\,& \Dc{\vsig} u(\xi) + \smfrac12
    \Dc{\vsig}^2 u(\xi) - \Dc{\vsig} u(x) + E_2' \\
    \label{eq:stress:c2_10}
    =\,& \Dc{\xi-x} \Dc{\vsig} u(x) + \smfrac12 \Dc{\vsig}^2 u(x) + E_2,
    \\[1mm]
    \notag
    \text{where} \quad |E_2| \lesssim\,& |\vsig| (|\vsig| + |\rho|)^2\,
    \delta_{\xi,\vsig},
    %
  \end{align}
  and combining \eqref{eq:stress:c2_10} with \eqref{eq:stress:c2_8}
  yields
  \begin{align}
    \label{eq:stress:c2_12}
    \Es_{\xi,\rho}  - V_{\rho} =\,& \sum_{\vsig \in \Rg}V_{\rho\vsig}
    \b[ \;\cdot\;, \Dc{\xi-x}
    \Dc{\vsig} u(x) + \smfrac12 \Dc{\vsig}^2 u(x) \b] + E_1 + E_3, \\
    \notag
    \text{where} \quad |E_3| \leq\,& \sum_{\vsig\in\Rg}
    \frac{m(\rho,\vsig)}{|\rho||\vsig|} |E_2|
    \lesssim
    \sum_{\vsig \in \Rg} \frac{m(\rho,\vsig)}{|\rho|} (|\vsig| + |\rho|)^2
    \, \delta_{\xi,\vsig},
  \end{align}
  which we insert into \eqref{eq:stress:c2_5} to obtain
  \begin{align}
    \label{eq:stress:c2_14}
    \Sa - \Scb
    =\,& \sum_{\rho \in \Rg} \sum_{\xi \in \L} \sum_{\vsig \in \Rg}
    \B\{ V_{\rho\vsig}\b[\;\cdot\;, \Dc{\xi-x}
    \Dc{\vsig} u + \smfrac12 \Dc{\vsig}^2 u \b] \otimes \rho \B\}
    \ww_{\xi,\rho}(x) + E_4, \\
    \notag
    \text{where} \quad E_4 =\,& \sum_{\rho \in \Rg} \sum_{\xi \in \L}
    \b[(E_1+E_3) \otimes \rho\b] \ww_{\xi,\rho}(x).
  \end{align}

  Rearranging the sums, we arrive at the expression
  \begin{equation}
    \label{eq:stress:c2_15}
    \Sa - \Scb
    = \sum_{\rho \in \Rg} \sum_{\vsig \in \Rg} \B\{
    V_{\rho\vsig}\b[\;\cdot\;, {\textstyle \sum_{\xi\in\L}} \ww_{\xi,\rho}(x) (\Dc{\xi-x}
    \Dc{\vsig} u + \smfrac12 \Dc{\vsig}^2u) \b] \otimes \rho \B\}
    + E_4.
  \end{equation}
  Using \eqref{eq:stress:phirho_Prop0} we see that
  \begin{equation}
    \label{eq:stress:c2_17}
    \sum_{\xi \in \L} \ww_{\xi,\rho}\, \Dc{\vsig}^2 u
    = \Dc{\vsig}^2 u, \quad \text{and} \quad
    \sum_{\xi\in\L} \ww_{\xi,\rho}\, \D_{\xi-x} \D_\vsig u =
    -\smfrac12 \D_\rho \D_\vsig u.
  \end{equation}
  Combining \eqref{eq:stress:c2_17} with \eqref{eq:stress:c2_15}, we
  arrive at the identity
  \begin{align*}
    \Sa - \Scb =\,& \frac12 \sum_{\rho \in \Rg} \sum_{\vsig \in \Rg}
    V_{\rho\vsig} \b[\;\cdot\;, \Dc{\vsig}^2 y - \Dc{\vsig}\Dc{\rho} y \b]
    \otimes \rho + E_4.
  \end{align*}
  Applying the symmetry $V_{\rho\vsig} = V_{-\rho,-\vsig}$
  (cf. \S\ref{sec:intro:V:symmetry}), yields
  \begin{align*}
    V_{\rho\vsig} \b[\;\cdot\;, \Dc{\vsig}^2 y - \Dc{\vsig}\Dc{\rho} y \b]
    \otimes \rho
    + V_{-\rho,-\vsig} \b[\;\cdot\;, \Dc{-\vsig}^2 y - \Dc{-\vsig}\Dc{-\rho} y \b]
    \otimes (-\rho) = 0,
  \end{align*}
  and hence we deduce that $\Sa - \Scb = E_4$.

  It remains to bound $E_4$. Using its definition
  \eqref{eq:stress:c2_14}, and the bounds \eqref{eq:stress:c2_8} for
  $E_1$ and \eqref{eq:stress:c2_12} for $E_3$, and the estimate
  $\sum_{\xi\in\L} \ww_{\xi,\rho} f_\xi \leq \max_{\xi\in\L,
    \ww_{\xi,\rho}(x) \neq 0} f_\xi$ we estimate
  \begin{align}
    \notag
    |E_4| \leq\,& \sum_{\rho\in\Rg} |\rho|
    \sum_{\xi\in\L} (|E_1|+|E_3|) \ww_{\xi,\rho} \\
    \notag
    \lesssim\,& \sum_{\rho,\vsig,\tau\in\Rg} (|\rho|+|\vsig|+|\tau|)^2
    m(\rho,\vsig,\tau)
    \max_{\substack{\xi\in\L\\ \chi_{\xi,\rho}(x) \neq 0}}
    \eps_{\xi,\vsig} \eps_{\xi,\tau} \\
    \label{eq:stress:c2_25}
    & + \sum_{\rho,\vsig\in\Rg}
    (|\rho|+|\vsig|)^2  m(\rho,\vsig)
    \max_{\substack{\xi\in\L\\ \chi_{\xi,\rho}(x) \neq 0}}
    \delta_{\xi,\vsig}.
  \end{align}

  There exists a constant $c > 0$ such that, if $\chi_{\xi,\rho}(x)
  \neq 0$, then $|\xi + t\rho - x| \leq c$ for some $t \in [0, 1]$ and
  one readily checks that this implies
  \begin{displaymath}
    \nu_{\xi,\vsig} - x \subset \nu_{\rho,\vsig} := \b\{ y \in \R^d \bsep {\rm
      dist}(y, \conv\{\pm 2\rho\pm\vsig\}) \leq 2 c \b\}.
  \end{displaymath}
  Hence, we obtain
  \begin{displaymath}
    \max_{\substack{\xi\in\L \\ \chi_{\xi,\rho}(x) \neq 0}} \eps_{\xi,\vsig}
    \eps_{\xi,\tau} \leq \| \D^2 u \|_{L^\infty(x+\nu_{\rho,\vsig})} \|
    \D^2 u \|_{L^\infty(\nu_{\rho,\tau})} \quad \text{ and } \quad
    \max_{\substack{\xi\in\L \\ \chi_{\xi,\rho}(x) \neq 0}} \delta_{\xi,\vsig}
    \leq \| \D^3 u \|_{L^\infty(x+\nu_{\rho,\vsig})}.
  \end{displaymath}
  Inserting these bounds into \eqref{eq:stress:c2_25}, and recalling
  that $\Sa - \Scb = E_4$, we obtain the stated estimate.

  The statements about the sets $\nu_{\rho,\vsig}$ are easy to
  establish.
\end{proof}

\subsection{Global modeling error estimate}
\label{sec:stress:moderr}
We present a modeling error estimate that is a natural corollary of
Theorem \ref{th:stress:c2}. For our subsequent analysis we only
require the case $p = 2$, however, we give a more general statement
since the same proof applies verbatim for general $p$. Earlier results
this direction have been obtained in \cite{MaSu:2011, E:2007a}.

\begin{lemma}
  \label{th:stress:moderr_negSob}
  Let $u \in \WWh^{3,p} \cap \Bc$, $p \in (1,\infty]$, and let $\tilu
  := \zz \ast u$; then, for all $v \in \Ush^{1,p'}$,
  \begin{equation}
    \label{eq:moderr_negSob}
    \b| \b\<\del\Ea(\tilu), \tilv \b\> - \b\< \del\Ec(u), \barv \b\>
    \b|
    \lesssim \b( \Ms^{(2,p)} \| \D^3 u \|_{\LL^p}
    + \Ms^{(3,p)} \| \D^2 u \|_{\LL^{2p}}^{2} \b) \|\D\barv \|_{\LL^{p'}},
  \end{equation}
  where the constants $\Ms^{(j, p)}$ are defined in \S\ref{sec:intro:V:decya}.
\end{lemma}
\begin{proof}
  The case $p = \infty$ follows immediately from Theorem
  \ref{th:stress:c2}, hence we assume that $p \in (1,
  \infty)$. Further, we assume that $v \in \Usz$, and apply a density
  argument to obtain the general statement.

  First, we need to show that $w := \tilu|_\L \in \Ba$. To that end,
  we estimate
  \begin{align*}
    |D_\rho w(\xi)| = \bg|\int_{t = 0}^1 \D_\rho \tilu(\xi+t\rho)
    \dt\bg| \leq |\rho| \|\D \tilu \|_{L^\infty} \leq |\rho|\|
    \zz\|_{L^1} \| \D u \|_{L^\infty}.
  \end{align*}
  Since $u \in \Bc$ and $\| \zz \|_{L^1} = \int \zz \dx = 1$, we
  obtain that $|D_\rho w(\xi)| \leq \inv |\rho|$ and hence $w =
  \tilu|_\L \in \Ba$. Hence it follows that the first variations on
  the left-hand side are well-defined.

  From Proposition \ref{th:Sa} it follows that
  \begin{align}
    \notag
    \b| \b\<\del\Ea(\tilu), \tilv \b\> - \b\< \del\Ec(u), \barv \b\> \b|
    =\,& \int_\Om \b[ \Sa(\tilu) - \Scb(u) \b] : \D \barv \dx \\
    \label{eq:stress:moderr:10}
    \leq\,& \b\| \Sa(\tilu) - \Scb(u) \b\|_{\LL^p} \| \D \barv
    \|_{\LL^{p'}}.
  \end{align}
  We apply the triangle inequality,
  \begin{equation}
    \label{eq:stress:moderr:15}
    \b\| \Sa(\tilu) - \Scb(u) \b\|_{\LL^p} \leq
    \b\| \Sa(\tilu) - \Scb(\tilu) \b\|_{\LL^p}
    + \b\| \Scb(\tilu) - \Scb(u) \b\|_{\LL^p},
  \end{equation}
  to separately estimate the two terms on the right-hand side.

  Applying \eqref{eq:bound_delj_Ec} and Lemma
  \ref{th:convolution_error}, we can bound the second term by
  \begin{equation}
    \label{eq:stress:moderr:20}
    \b\| \Scb(\D\tilu) - \Scb(\D u) \b\|_{\LL^p} \lesssim M^{(2)} \| \D
    \tilu - \D u \|_{\LL^p} \lesssim M^{(2)} \| \D^3 u \|_{L^p}.
  \end{equation}

  To estimate the first term on the right-hand side of
  \eqref{eq:stress:moderr:15} we first apply Theorem
  \ref{th:stress:c2} to obtain
  \begin{align}
    \label{eq:stress:moderr:100}
    \b| \Sa(\tilu; x) - \Scb(\tilu; x) \b| \lesssim\,&
    \sum_{\rho,\vsig\in\Rg} \ms^{(\infty)}(\rho,\vsig) \|
    \D^3 \tilu \|_{L^\infty(x+\nu_{\rho,\vsig})} \\
    \notag
    & +
    \sum_{\rho,\vsig,\tau\in\Rg} \ms^{(\infty)}(\rho,\vsig,\tau)
    \| \D^2 \tilu \|_{L^\infty(x+\nu_{\rho,\vsig})} \| \D^2 \tilu \|_{L^\infty(x+\nu_{\rho,\tau})}.
  \end{align}
  Next, we estimate the $L^p$-norm of the first term on the right-hand
  side. To that end, we first recall the definition of
  $\nu_{\rho,\vsig}$ from Theorem \ref{th:stress:c2} as well as the
  enlarged sets $\nu'_{\rho,\vsig}$ defined in
  Lemma~\ref{th:conv_est}. Let $(w_{\rho\vsig}) \in \ell^p(\Rg^2)$,
  and let $w := (\sum_{\rho,\vsig\in\Rg}
  w_{\rho\vsig}^{p'}\ms^{(\infty)}(\rho,\vsig))^{p'/p}$, then applying
  first H\"{o}lder's inequality and then Lemma \ref{th:conv_est},
  gives
  \begin{align*}
    \int_\Om \bg| \sum_{\rho,\vsig\in\Rg} \ms^{(\infty)}(\rho,\vsig) \|
    \D^3 \tilu \|_{L^\infty(\nu_{\rho,\vsig})} \bg|^p \dx
    \leq\,& w \sum_{\rho,\vsig \in \Rg} w_{\rho\vsig}^{-p}
    \ms^{(\infty)}(\rho,\vsig) \int_\Om
    \| \zz \ast \D^3 u \|_{L^\infty(x+\nu_{\rho,\vsig})}^p \dx \\
    \leq\,& w \sum_{\rho,\vsig\in\Rg} w_{\rho\vsig}^{-p}
    \ms^{(\infty)}(\rho,\vsig) {\rm vol}(\nu_{\rho,\vsig}')\,\| \D^3 u \|_{L^p}^p.
  \end{align*}
  Choosing $w_{\rho\vsig}$ to balance $w$ with
  $\sum_{\rho,\vsig\in\Rg} w_{\rho\vsig}^{-p}
  \bar{M}^{(\rho,\vsig)}_\mu \,\vol(\nu_{\rho,\vsig}')$, and noting that
  $\vol(\nu'_{\rho\vsig}) \lesssim \vol(\nu_{\rho\vsig})$ (this follows
  immediately from the definition of $\nu_{\rho\vsig}$) yields
  \begin{equation}
    \label{eq:stress:moderr:110}
    \bg(\int_\Om \bg| \sum_{\rho,\vsig\in\Rg} \ms^{(\infty)}(\rho,\vsig) \|
    \D^3 \tilu \|_{L^\infty(\nu_{\rho,\vsig})} \bg|^p \dx\bg)^{1/p} \lesssim
    \Ms^{(2,p)} \| \D^3 u \|_{L^p}.
  \end{equation}
  With an analogous argument we obtain
  \begin{equation}
    \label{eq:stress:moderr:120}
    \bg(\int_\Om \bg| \sum_{\rho,\vsig,\tau\in\Rg} \ms^{(\infty)}(\rho,\vsig,\tau) \|
    \D^2 \tilu \|_{L^\infty(x+\nu_{\rho,\vsig})} \|
    \D^2 \tilu \|_{L^\infty(x+\nu_{\rho,\tau})} \bg|^{2p} \dx\bg)^{1/p} \lesssim
    \Ms^{(3,p)} \| \D^2 u \|_{L^{2p}}^2.
  \end{equation}
  Combining \eqref{eq:stress:moderr:100}, \eqref{eq:stress:moderr:110}
  and \eqref{eq:stress:moderr:120} with \eqref{eq:stress:moderr:20},
  and noting that $M^{(2)} \leq \Ms^{(2,p)}$, completes the proof.
\end{proof}

\section{Elastostatic problems}
\label{sec:errsm}
In this section we present error estimates for local minimizers of the
Cauchy--Born model. We essentially recover the result of E and Ming
\cite[Thm. 2.3]{E:2007a} for a more general class of interactions, and
in the more challenging setting of an infinite domain and infinite
interaction range. Moreover, due to our new consistency estimates in
\S\ref{sec:stress}, we obtain sharper and more explicit estimates.

\subsection{The variational problems}
\label{sec:min:problems}
\subsubsection{Continuous external forces}
\label{sec:cont_ext_frcs}
For $f, g \in L^1_{\rm loc}$ with $f\cdot g \in L^1$ we define the inner
product
\begin{displaymath}
  (f, g)_{\R^d} := \int_{\R^d} f\cdot g \dx.
\end{displaymath}
We say that $f \in L^{1}_{\rm loc} \cap \WWh^{-1,2}$ if there exists a
constant $\| f \|_{\WWh^{-1,2}}$ such that
\begin{displaymath}
  (f, v)_{\R^d} \leq \| f \|_{\WWh^{-1,2}} \| \D v \|_{L^2} \qquad
  \forall v \in C^\infty_0.
\end{displaymath}
In this case there exists a unique continuous extension of $(f,
\;\cdot\;)_{\R^d}$ to $\WWh^{1,2}$.

\subsubsection{The Cauchy--Born Problem}
\label{sec:cb_problem}
In the Cauchy--Born model, given $f^\c \in L^1_{\rm loc} \cap
\WWh^{-1,2}$, we seek
\begin{equation}
  \label{eq:intro:var_c}
  u^\c \in \arg \min \b\{ \Ec(u) - ( f^\c, u)_\Om
  \bsep u \in \WWh^{1,2} \b\}.
\end{equation}
We understand \eqref{eq:intro:var_c} as a {\em local} minimization
problem with respect to the $(\WWh^{1,2} \cap
\WWh^{1,\infty})$-topology. If $u^\c \in \Bc$ is a solution to
\eqref{eq:intro:var_c}, then it satisfies the first-order optimality
condition
\begin{equation}
  \label{eq:intro:var_c_EL}
  \b\< \del\Ec(u^\c), v \b\> = ( f^\c, v)_\Om \qquad \forall v \in \WWhz.
\end{equation}
We call a solution $u^\c$ of \eqref{eq:intro:var_c_EL} {\it stable} if
there exists $c_0 > 0$ such that
\begin{equation}
  \label{eq:intro:var_c_suff}
  \b\< \ddel\Ec(u^\c) v, v \b\> \geq c_0 \| \D v \|_{\LL^2}^2 \qquad
  \forall v \in \WWhz.
\end{equation}
From Proposition \ref{th:defn_Ec} it follows that, if $u^\c$ is a
stable solution of \eqref{eq:intro:var_c_EL}, then it is a strict
$(\WWh^{1,2} \cap \WWh^{1,\infty})$-local minimizer of $\Ec - ( f^\c,
\;\cdot\;)_\Om$, and hence a solution of \eqref{eq:intro:var_c}.

\subsubsection{External forces in the atomistic problem}
\label{sec:min:ext_frcs}
For $f, g \in \Us$ with $f \cdot g \in \ell^1$ we define the inner
product
\begin{displaymath}
  (f, g)_{\L} := \sum_{\xi \in \L} f(\xi) \cdot g(\xi).
\end{displaymath}
We say that $f \in \Ush^{-1,2}$ if $f \in \Us$ and there exists a
constant $\| f \|_{\Ush^{-1,2}}$ such that
\begin{displaymath}
  (f, v)_{\L} \leq \| f \|_{\Ush^{-1,2}} \, \| \D \barv \|_{L^2}
  \qquad \forall v \in \Usz.
\end{displaymath}
In this case there exists a unique continuous extension of $(f,
\;\cdot\;)_{\L}$ to $\Ush^{1,2}$.

\subsubsection{The atomistic problem}
\label{sec:min:atm_prob}
Given $f^\a \in \Ush^{-1,2}$ we seek
\begin{equation}
  \label{eq:intro:var_a}
  u^\a \in \arg\min \b\{ \Ea(u) - (f^\a, u)_\L \bsep u \in \Ush^{1,2} \b\}.
\end{equation}
We understand \eqref{eq:intro:var_a} as a {\em local} minimization
problem. If $u^\a \in \Ba$ is a solution to \eqref{eq:intro:var_a},
then it satisfies the first-order optimality condition
\begin{equation}
  \label{eq:intro:var_a_EL}
  \b\< \del\Ea(u^\a), v \b\> = (f^\a, v)_\L \qquad \forall v \in \Usz.
\end{equation}
We call a solution $u^\a$ of \eqref{eq:intro:var_a_EL} {\em stable} if
there exists $c_0 > 0$ such that
\begin{equation}
  \label{eq:intro:var_a_suff}
  \b\< \ddel\Ea(u^\a) v, v \b\> \geq c_0 \| \D \barv\|_{L^2}^2
  \qquad \forall v \in \Usz.
\end{equation}
From Proposition \ref{th:defn_Ea} it follows that, if $u^\a$ is a
stable solution of \eqref{eq:intro:var_a_EL} then $u^\a$ is a strict
$\Ush^{1,2}$-local minimizer of $\Ea - ( f^\a, \;\cdot\;)_\L$
and hence a solution of \eqref{eq:intro:var_a}.

\subsection{Stability of small displacements}
\label{sec:stab_small}
We say that the lattice $\mA \cdot \L$ is stable if
\begin{equation}
  \label{eq:defn_Lam}
  \Lam := \inf_{\substack{v \in \Usz \\ \| \D\barv \|_{L^2} = 1}}
  \b\< \ddel\Ea(0) v, v \b\> > 0.
\end{equation}
Physically, \eqref{eq:defn_Lam} states that small distortions of the
Lattice $\mA \cdot \L$ {\em increase} its energy.

For simple interactions \eqref{eq:defn_Lam} can be proven analytically
\cite{OrtnerShapeev:2011pre, EM:cb_dynamic}. In practise one checks
this stability condition by block-diagonalising $\ddel\Ea(0)$ using
Fouriers series \cite{Wallace98, E:2007a, Hudson:stab}, that is, one
checks whether the dispersion relation satisfies $\omega(k) \geq c
|k|$. The condition is discussed in more detail in \cite{Hudson:stab}
and in Appendix \ref{sec:app_stab}.

Assuming only \eqref{eq:defn_Lam} we can deduce stability of ``small''
displacements both in the atomistic and Cauchy--Born models. The
factor $\smfrac12$ in the following result is arbitrary and may be
replaced with any number between zero and one.

\begin{proposition}
  \label{th:stab_small}
  Let $\mA \cdot \L$ be stable, then there exists $\inv_1 >
  0$ such that, for $\inv \leq \inv_1$,
  \begin{align*}
    \b\< \ddel\Ea(u) v, v \b\> \geq~& \smfrac12 \Lam \|\D \barv
    \|_{L^2}^2 \qquad \forall v \in
    \Ush^{1,2}, \quad \forall u \in \Ba, \quad \text{and} \\
    \b\< \ddel\Ec(u) v, v \b\> \geq~& \smfrac12 \Lam \|\D v\|_{L^2}^2
    \qquad \forall v \in \WWh^{1,2}, \quad \forall u \in
    \Bc.
  \end{align*}
\end{proposition}

Before we prove Proposition \ref{th:stab_small} we state a variant of
a classical intermediate result (see, e.g., \cite[p. 89]{Wallace98};
the following proof is adapted from \cite[Thm. 3.1]{Hudson:stab}).

\begin{lemma}
  \label{th:stab_hom}
  Let $\Lam$ be defined by \eqref{eq:defn_Lam}, then
  \begin{displaymath}
    \b\<\ddel\Ec(0) v, v \b\> \geq \Lam \|\D v \|_{L^2}^2 \qquad
    \forall v \in \WWh^{1,2}.
  \end{displaymath}
\end{lemma}
\begin{proof}
  Fix $v \in \WWhz \setminus \{0\}$ and set $v_N(x) := N v(N^{-1} x)$
  for any $N \in \N$, and let $w_N := v_N|_\L$; then we have
  \begin{displaymath}
    \< \ddel\Ec(0) v, v \> = N^{-d} \< \ddel\Ec(0) v_N, v_N \>.
  \end{displaymath}
  Taking into account \cite[Remark 1.1.1]{Hudson:stab}, and using the
  fact that $v$ is smooth, Lemma 3.2 in \cite{Hudson:stab} yields
  \begin{displaymath}
    N^{-d} \b| \< \ddel\Ec(0) v_N, v_N \> - \< \ddel \Ea(0) w_N,
    w_N \> \b| \to 0 \quad \text{as } N \to \infty.
  \end{displaymath}
  We remark, that \cite[Lemma 3.2]{Hudson:stab} is formulated for
  finite-range interactions only, however, under the assumption that
  $M^{(2)}$ is finite a straightforward approximation argument extends it
  to the present case.

  Using the smoothness of $v$ it is also easy to see that $N^{-d/2} \|
  \D w_N \|_{L^2} \to \|\D v\|_{L^2}$. Hence, we obtain
  \begin{displaymath}
    \Lam \leq \frac{\<\ddel\Ea(0) w_N, w_N\>}{\| \D w_N \|_{L^2}^2}
    \overset{N \to \infty}{\longrightarrow}
    \frac{ \< \ddel \Ec(0) v, v
      \>}{\|\D v\|_{L^2}^2}.
  \end{displaymath}
  Taking the infimum over all $v \in \WWhz$ yields the stated result.
\end{proof}

\begin{proof}[Proof of Proposition \ref{th:stab_small}]
  We first consider the atomistic case. A simple variation of the
  proof of \eqref{eq:bound_delj} with $j = 3$ (replacing $\|\D u
  \|_{L^\infty}$ with $\max_{\xi\in\L} \max_{\rho\in\Rg} |D_\rho
  u(\xi)|$) gives the Lipschitz bound
  \begin{displaymath}
    \b|\b\< (\ddel\Ea(u) - \ddel\Ea(0)) v, v \b\>| \leq c M^{(3)} \inv \| \D v \|_{L^2}^2 \qquad \forall v \in \Ush^{1,2},
  \end{displaymath}
  where $c$ is a generic positive constant. Hence, choosing $\inv \leq \gamma /
  (2 c M^{(3)})$ yields the atomistic stability result.

  After employing Lemma \ref{th:stab_hom}, the proof for the
  continuous case is analogous.
\end{proof}

\begin{remark}
  We have shown that stability of the atomistic model implies
  stability of the continuum model, using only pointwise convergence
  of the atomistic hessian to the continuum hessian (this is in fact a
  consequence of convergence of the energy) and scale-invariance of
  the continuum limit. Conversely, one can construct examples
  \cite{EM:cb_dynamic, Hudson:stab} where the continuum limit is
  stable (in 1D, convex) while the atomistic model is not stable in
  the sense of \eqref{eq:defn_Lam}. In this case, we would still
  expect that atomistic solutions of both the static and dynamic
  problem exist (in a suitable extended framework), however, we can no
  longer expect them to be ``close'' to the solutions of the
  Cauchy--Born equations. We give a more detailed discussion in
  Appendix \ref{sec:app_stab}, from which can conclude that
  \eqref{eq:defn_Lam} (or a similar assumption) is also {\em
    necessary} to obtain the results we seek.
\end{remark}

\subsection{Main result}
\label{sec:errsm:result}
We first restate the Cauchy--Born equation \eqref{eq:intro:var_c} at a
macroscopic scale
\begin{equation}
  \label{eq:errsm:cb_scale}
  X = \eps x, \quad U = \eps u, \quad \text{and} \quad F^\c = \eps^{-1} f^\c,
\end{equation}
where $X \in \R^d, U, F^\c : \R^d \to \R^d$, and $\eps$ is the atomic
spacing in the $X$-scale. In these {\em macroscopic variables}, the
Cauchy--Born equation \eqref{eq:intro:var_c_EL} reads, formally
\begin{equation}
  \label{eq:errsm:cb_eqn_scaled}
  - {\rm div}_X \Scb(\D_X U^\c) = F^\c,
\end{equation}
By assuming that $F^\c$ is small, more precisely, that
\begin{equation}
  \label{eq:Fc_small}
  \| F^\c \|_{\WWh^{-1,2}} + \| \D F^\c \|_{L^2} =: \delta
\end{equation}
is sufficiently small, we will be able to prove that there exists a
solution $U^\c$ to \eqref{eq:errsm:cb_eqn_scaled}. Reversing the
scaling \eqref{eq:errsm:cb_scale}, we obtain a solution $u^\c(x) :=
\eps^{-1} U^\c(\eps x)$ of the atomic scale equation
\eqref{eq:intro:var_c_EL}, with external force $f^\c(x) := \eps
F^\c(\eps x)$. We note that
\begin{equation}
 \label{eq:bnd_y_f}
  \begin{split}
    \| \D^3 u^\c \|_{\LL^2} +   \| \D^2 u^\c \|_{\LL^4}^2 =\,& \eps^{2-d/2}
    \b( \| \D_X^3 U^\c \|_{\LL^2} + \| \D_X^2 U^\c \|_{\LL^4}^2 \b),
  \end{split}
\end{equation}
which implies that the atomistic/Cauchy--Born modelling error in the
internal forces is of order $O(\eps^{2-d/2}$ (cf. Lemma
\ref{th:stress:moderr_negSob}). To ensure that the modelling error in
the external forces is of the same order of magnitude, we shall assume
that
\begin{equation}
  \label{eq:asm_err_extfrc}
  \b| (f^\c, v)_{\R^d} - (f^\a, v)_{\L} \b| \leq C_f \delta
  \eps^{2-d/2}.
\end{equation}
As a concrete example, we show in Lemma \ref{th:ex_frc} that, if $F^\c
\in \WWh^{-1,2} \cap W^{1,2}$, $f^\c$ is defined by
\eqref{eq:errsm:cb_scale}, and $f^\a$ is defined by $f^\a(\xi) := \int
\zz(x-\xi) f^\c(x) \dx$, then $f^\a \in \Ush^{-1,2}$ and
\eqref{eq:asm_err_extfrc} holds.

\begin{theorem}
  \label{th:errsm:mainthm}
  Let $d \leq 3, k \geq 4$, and suppose that $\mA\cdot\L$ is stable
  and that $\Ms^{(2, 2)}$ and $\Ms^{(3, 2)}$ are finite
  (cf. \S~\ref{sec:intro:V:decya}).

  There exist constants $\delta_0, \eps_0 > 0$ such that, for $F^\c
  \in \WWh^{-1,2} \cap W^{1,2}$ satisfying \eqref{eq:Fc_small}, $f^\c$
  defined by \eqref{eq:errsm:cb_scale} and $f^\a \in \Ush^{-1,2}$
  satisfying \eqref{eq:asm_err_extfrc}, and for $\delta \leq \delta_0$
  and $\eps \leq \eps_0$, there exist stable solutions $u^\c$ and
  $u^\a$ of, respectively, \eqref{eq:intro:var_c_EL} and
  \eqref{eq:intro:var_a_EL}, such that
  \begin{displaymath}
    \eps^{d/2} \| \D u^\c - \D \Itil u^\a \|_{L^2} \leq C \frac{\delta \eps^2}{\gamma},
  \end{displaymath}
  where $C = C_{f} C(\Ms^{(2,2)}/\Lam, \Ms^{(3,2)}/\Lam)$.
\end{theorem}

\begin{remark}
  1. Formally, Theorem \ref{th:errsm:mainthm} states that, if the
  external forces are sufficiently small and of a ``macroscopic
  nature'' (encoded in the assumption that $f^\c(x) = \eps F^\c(\eps
  x)$, which implies that $\D f^\a \approx \D f^\c = O(\eps^2)$), then
  the atomistic solution may be approximated to second-order accuracy
  by a solution of the Cauchy--Born model.

  2. The conclusion of Theorem \ref{th:errsm:mainthm} may also be
  stated as
  \begin{displaymath}
    \| \D u^\c - \D \Itil u^\a \|_{L^2} \lesssim
    C' \smfrac{\delta}{\gamma} \b( \Ms^{(2,2)} \| \D^3
    u^\c \|_{L^2} + \Ms^{(3,2)} \| \D^2 u^\c \|_{L^4}^2 \b),
  \end{displaymath}
  where $C' = C'(\delta_0)$.  In macroscopic units, with $U^\a_\eps(X)
  := \eps \Itil u^\a(\eps^{-1} X)$ the estimate reads
  \begin{displaymath}
    \| \D U^\c - \D U^\a_\eps \|_{L^2} \leq C' \smfrac{\delta}{\gamma}
    \eps^2 \b( \Ms^{(2,2)} \| \D^3
    U^\c \|_{L^2} + \Ms^{(3,2)} \| \D^2 U^\c \|_{L^4}^2 \b).
  \end{displaymath}
  On the right-hand sides of both of these estimates we may also
  replace $u^\c$ with $\Itil u^\a$, respectively $U^\c$ with $U^\a$,
  which effectively turns them into {\it a priori} error estimates.

  3. The factor $\delta$ in these estimates shows that the error is
  $O(\eps^2)$ {\em relative} to the magnitude of the external force
  and hence the displacement; that is, our estimates are in fact
  relative error estimates.
\end{remark}

\medskip \noindent The proof of this result uses a quantitative
version of the inverse function theorem. The following version is
taken from \cite[Thm. 2.1]{ortner_apostex}.

\def\iF{\mathscr{F}}
\def\iX{\mathscr{A}}
\def\iY{\mathscr{B}}
\def\iA{\mathscr{O}}
\begin{lemma}[Inverse Function Theorem]
  \label{th:errsm:inverse_fcn_thm}
  Let $\iX, \iY$ be Banach spaces, $\iA$ an open subset of $\iX$, and
  let $\iF : \iA \to \iY$ be Fr\'{e}chet differentiable. Suppose also
  that there exist $\eta, \sigma > 0$ and a monotone function $\omega
  : [0, +\infty) \to [0, +\infty]$ such that
  \begin{align*}
    & \overline{B_{\iX}(0, 2\eta\sigma)} \subset \iA, \quad  \| \iF(0) \|_{\iY} \leq \eta, \quad
    \| \del\iF(0)^{-1} \|_{L(\iY,\iX)} \leq \sigma, \\
    & \| \iF'(U) - \iF'(V) \|_{L(\iX,\iY)} \leq \omega\b( \|U - V
    \|_{\iX} \b)  \quad
    \text{for} \quad \|U\|_{\iX}, \|V\|_\iX \leq 2 \eta \sigma, \\
    & 2\sigma \omega(2\eta\sigma) \leq 1, \quad
    \text{and} \quad \sigma \omega(2\eta\sigma) < 1.
  \end{align*}
  Then, there exists a unique $U \in \iX$ such that $\iF(U) = 0$ and
  $\|U\|_{\iX} \leq 2\eta\sigma$.
\end{lemma}
\begin{proof}
  The result follows from \cite[Thm. 2.1]{ortner_apostex}, upon
  replacing $\sigma$ with $\sigma^{-1}$, taking $\bar{\omega}(t) = t
  \omega(t)$ (admissible since $\omega$ is monotone), and $R = 2\eta
  \sigma$.
\end{proof}

\begin{proof}[Proof of Theorem \ref{th:errsm:mainthm}]
  {\it Part 1: Existence of a Cauchy--Born solution: }
  Since $\mA\cdot\L$ is stable, Lemma \ref{th:stab_small} implies that
  $\ddel\Ec(0)$ is positive definite, which is equivalent to the
  statement that $\D^2 W(\mO)$ satisfies the strong Legendre--Hadamard
  condition. Under this condition it is proven in
  \cite{OrSu:homsob:2012} that $\ddel\Ec(0) : \WWh^{3,2} \to
  \WWh^{1,2} \cap \WWh^{-1,2}$ is an isomorphism. Hence, we can hope
  to apply Lemma \ref{th:errsm:inverse_fcn_thm} with
  \begin{align*}
    & \iX :=\WWh^{3,2}, \quad \iY :=  \WWh^{1,2} \cap \WWh^{-1,2},  \\
    & \iF(U) := \del\Ec(U) - F^\c, \quad \text{and} \quad \iA :=
    \{U \in \iX \sep \|U\|_{\WWh^{3,2}} \leq \delta \}.
  \end{align*}
  Since $d \leq 3$, $\iX$ is embedded in $\WWh^{1,\infty}$ and hence,
  for $\delta$ sufficiently small, we have $\iA \subset \Bc$. Hence,
  \eqref{eq:bound_delj_Ec} implies that, for $U, V \in \iA$,
  \begin{displaymath}
    \b\| \iF(U+V) - \iF(U) - \ddel\Ec(U)[V,\;\cdot\;] \|_{\WWh^{-1,2}} \leq
    M^{(3)} \| \D V \|_{L^2}^2 \leq M^{(3)} \| V \|_{\iX}^2.
  \end{displaymath}
  A tedious but straightforward computation also shows that
  \begin{align*}
    & \b\| \iF(U+V) - \iF(U) - \ddel\Ec(U)[V,\;\cdot\;] \|_{\WWh^{1,2}} \\
    =\, & \b\| \D {\rm div} \b[ \D W(\D U+\D V) - \D W(\D U) - \D^2 W(\D U)
    : \D V \b] \b\|_{L^2} \\
    \lesssim\,& o_1(\| V \|_{\iX}) \qquad \text{for } U \in \iA, \text{ and for } \| V \|_{\iX} \text{
      sufficiently small,}
  \end{align*}
  where $o_1(t) \ll t$ as $t \to 0$. (The function $o_1(t)$ depends on
  $M^{(3)}$ and $M^{(4)}$ and on the modulus of continuity of $\D^4 W$ in
  the set $\{ \mF \sep |\mF| \leq \inv \}$; if $k = 5$, then $o_1(t) =
  \sum_{j = 3}^5 M^{(j)} t^2$.) This shows that $\iF$ is Fr\'echet
  differentiable and $\del\iF(U) = \ddel\Ec(U)$.

  Similarly, we can also show that
  \begin{displaymath}
    \| \del\iF(U) - \del\iF(U') \|_{L(\iX, \iY)} \leq o_0( \| U - U' \|_{\iX}),
  \end{displaymath}
  where $o_0(t) \to 0$ as $t \to 0$. (In fact, $o_0(t) = o_1(t)/t$.)

  We have in particular established that $\del\iF(0) =
  \ddel\Ec(0)$, which we already know to be an isomorphism from
  $\iX$ to $\iY$. Moreover, by assumption we have
  \begin{displaymath}
    \| \iF(0) \|_{\iY}
    = \| F^\c \|_{\WWh^{-1,2}\cap \WWh^{1,2}} \leq \delta.
  \end{displaymath}
  Hence, Lemma \ref{th:errsm:inverse_fcn_thm} guarantees that, for
  $\delta$ sufficiently small, there exists $U^\c \in \WWh^{1,2} \cap
  \WWh^{3,2}$ satisfying \eqref{eq:errsm:cb_eqn_scaled} in the strong
  sense, and
  \begin{equation}
    \label{eq:a_priori_bound_Uc}
    \| U^\c \|_{\iX} \leq c \delta / \gamma,
  \end{equation}
  where $c$ is a generic constant. (The factor $1/\gamma$ is due to
  the fact that $\| \del\iF(0)^{-1}\|_{L(\iY, \iX)} \lesssim
  1/\gamma$.)

  Let $u^\c(x) := \eps^{-1} U^\c(\eps x)$, then the arguments given
  before the statement of the theorem, and Proposition
  \ref{th:stab_small}, show that $u^\c$ is a stable solution of
  \eqref{eq:intro:var_c_EL} with $f^\c$ given by
  \eqref{eq:errsm:cb_scale}.

  Upon noting that $\| \D^3 U^\c \|_{\LL^2} + \|\D^2 U^\c \|_{\LL^4}^2
  \lesssim \| U^\c \|_{\iX} \lesssim \delta/\gamma$, and that $\|\D
  U^\c \|_{L^\infty} \lesssim \| U^\c \|_{\iX} \lesssim
  \delta/\gamma$, we obtain the bounds
  \begin{equation}
    \label{eq:minthm:DYc_bnd}
    \| \D^3 u^\c \|_{L^2} + \| \D^2 u^\c \|_{L^4}^2
    \lesssim \eps^{2-d/2}\delta/\gamma, \quad \text{and} \quad
    \| \D u^\c \|_{L^\infty} \lesssim \delta/\gamma.
  \end{equation}
  %

  {\it Part 2: existence of an atomistic solution. } Recall the
  definition of $\tilde{u}^\c := \zz \ast u^\c$ from
  Lemma~\ref{th:stress:moderr_negSob}. We apply Lemma
  \ref{th:errsm:inverse_fcn_thm} with $\iX = \Ush^{1,2}$, $\iY =
  \Ush^{-1,2}$, $\iA := \{ w \in \Ush^{1,2} \sep \| \D \barw \|_{L^2}
  < \delta_1 \}$ for some constant $\delta_1 > 0$ that remains to be
  chosen. If $\delta_1$ is chosen sufficiently small, then
  $\tilde{u}^\c|_\L + \iA \subset \Ush^{1,2} \cap \Ba$, hence we can
  define
  \begin{displaymath}
    \b\<\iF(w), v \b\> := \b\< \del\Ea(\tilde{u}^\c + w), v \b\> - (f^\a,
    v)_{\L} \qquad \text{for } w \in \iA, \quad v \in \Ush^{1,2}.
  \end{displaymath}
  By Theorem \ref{th:defn_Ea}, $\iF$ is Fr\'echet differentiable in
  $\iA$, and $\del\iF$ is Lipschitz continuous in $\iA$, that is we
  can choose $\omega(t) = c M^{(3)} t$ in Lemma
  \ref{th:errsm:inverse_fcn_thm}.

  To obtain a stability estimate, we use \eqref{eq:minthm:DYc_bnd} and
  Proposition \ref{th:stab_small} to deduce that, if $\delta$ and
  $\delta_1$ are chosen sufficiently small, then
  \begin{displaymath}
    \b\< \ddel\Ea(\tilu^\c) v, v \> \geq \smfrac12 \Lam \| \D
    \barv \|_{L^2}^2 \qquad \forall v \in \Ush^{1,2},
  \end{displaymath}
  that is, $ \| \del\iF(0)^{-1} \|_{L(\iY, \iX)} \leq ( \smfrac12 \Lam
  \b)^{-1} =: \sigma$.

  To obtain a residual bound, we apply Lemmas
  \ref{th:stress:moderr_negSob} and \ref{th:ex_frc}, to estimate
  \begin{align*}
    \b\< \iF(0), \tilv \b\> =\,& \b\< \del\Ea(\tilu^\c), \tilv \b\> -
    (f^\a, \tilv)_{\L} \\
    =\,& \B\{\b\< \del\Ea(\tilu^\c), \tilv \b\> - \b\< \del\Ec(u^\c),
    \barv \b\> \B\}
    - \B\{ (f^\a, \tilv)_\L - (f^\c, \barv)_{\Om} \B\} \\
    \lesssim\,& \b( \Ms^{(2,2)} \| \D^3 u^\c \|_{L^2} + \Ms^{(3,2)} \| \D^2 u^\c
    \|_{L^4}^2 + \| \D f^\c \|_{L^2} \b) \| \D \barv\|_{L^2}.
  \end{align*}
  that is,
  \begin{displaymath}
    \|\iF(0) \|_{\iY} \leq \eta := C\, \b[1 + \gamma^{-1}
    (\Ms^{(2,2)} + \Ms^{(3,2)}) \b] \, \delta
    \eps^{2-d/2}.
  \end{displaymath}

  Lemma \ref{th:errsm:inverse_fcn_thm} states that, if
  \begin{equation}
    \label{eq:errsm:atm_ex_condition}
    C \frac{M^{(3)}}{\gamma} \bg( 1 + \frac{\Ms^{(2,2)} +
      \Ms^{(3,2)}}{\gamma} \bg) \frac{\delta}{\gamma} \eps^{2-d/2} < 1,
  \end{equation}
  then there exists a locally unique solution $v$ of $\iF(v) =
  0$. This can be guaranteed provided that $\delta \eps^{2-d/2}/\gamma
  < \eps_0 = \eps_0(\Ms^{(2,2)}/\gamma, \Ms^{(3,2)}/\gamma)$. (Recall that $M^{(3)}
  \leq \Ms^{(3,2)}$.)

  Let $w^\c := \tilu^\c|_\L$. Setting $u^\a(\xi) := w^\c(\xi) +
  v(\xi)$, and applying \eqref{eq:D_stab_Itil} we obtain the estimate
  \begin{displaymath}
    \| \D \baru^\a - \D \barw^\c \|_{L^2} \leq 2 \sigma \eta
    \lesssim \b[1 + \smfrac{1}{\Lam}
    (\Ms^{(2,2)} + \Ms^{(3,2)}) \b] \smfrac{\delta}{\Lam} \eps^{2-d/2}.
  \end{displaymath}
  Applying the interpolation error estimate given in Corollary
  \ref{th:interp_error} and \eqref{eq:D_stab_Itil}, we obtain
  \begin{align*}
    \| \D u^\c - \D \Itil u^\a \|_{L^2} \leq\,&
    \| \D \Itil(w^\c - u^\a) \|_{L^2} + \| \D u^\c - \D
    \Itil w^\c \|_{L^2} \\
    \lesssim\,& \| \D (u^\c - u^\a) \|_{L^2} + \| \D^3 u^\c \|_{L^2}
    \\
    \lesssim\,& \b[1 + \smfrac{1}{\Lam}
    (\Ms^{(2,2)} + \Ms^{(3,2)}) \b] \smfrac{\delta}{\Lam} \eps^{2-d/2}.
  \end{align*}
  This concludes the proof of the theorem.
\end{proof}

\section{Convergence to solutions of the Wave Equation}
\label{sec:wave}
In this section we consider the dynamic problem
\begin{equation}
  \label{eq:w:at}
  \begin{split}
    & (\ddot u^\a(t), v)_\L + \< \del\Ea(u^\a(t)), v\> = 0 \qquad \forall v \in
    \Usz, \quad t > 0, \\
    & u^\a(0) = u^\a_0, \quad \dot{u}^\a(0) = u^\a_1.
  \end{split}
\end{equation}
We will prove that, if the initial condition is
``macroscopic'', then there exists a unique solution to
\eqref{eq:w:at}, which remains close to a solution of the
corresponding Cauchy--Born wave equation for a ``macroscopic time
interval''.

For simplicity, we do not consider external forces in the the dynamic
problem.

\subsection{The macroscopic wave equation}
\label{sec:wave:cb}
Formally, the continuum limit of \eqref{eq:w:at} is the Cauchy--Born
wave equation
\begin{displaymath}
  \ddot u^\c(t) - {\rm div} \Scb(\D u^\c(t)) =0,
\end{displaymath}
subject to initial conditions. Upon rescaling
\begin{equation}
  \label{eq:w:cb_scale}
  X := \eps x,\quad U := \eps u, \quad T := \eps t,
\end{equation}
we formally obtain
\begin{equation}
  \label{eq:w:macro}
  \smfrac{{\rm d}^2}{{\rm d} T^2} U^\c(T) - {\rm div}_X \Scb\D_X U^\c(T)) =\,0,
\end{equation}
which we supply with the initial condition
\begin{equation}
  \label{eq:w:macro_init}
  U^\c(0) = U_0^\c, \quad \text{and} \quad
  \D_T U^\c(0) = U_1^\c.
\end{equation}

To establish well-posedness of \eqref{eq:w:macro},
\eqref{eq:w:macro_init}, we apply the well-established theory.  In our
context, Theorem III in \cite{HuKaMa} reads as follows. Note that,
from here on, we employ again the standard Sobolev spaces instead of
homogeneous Sobolev spaces.

\begin{proposition}
  \label{th:HuKaMa}
  Let $d \in \{1,2,3\}$, $k \geq 4$, and suppose that $\mA \cdot \L$
  is stable. Let $U_0^\c \in W^{4,2}, U_1^\c \in W^{3,2}$ with $\| \D
  U_0^\c \|_{L^\infty} < \kappa$, then there exists $T^\c > 0$ such
  that the system \eqref{eq:w:macro}, \eqref{eq:w:macro_init} has a
  unique solution $U^\c \in C^2([0, T^\c]; W^{2,2}) \cap C^1([0,
  T^\c]; W^{3,2}) \cap C([0, T^\c]; W^{4,2})$, satisfying $\max_{T \in
    [0, T^\c]} \| \D U^\c(T)\|_{L^\infty} < \inv$.
\end{proposition}
\begin{proof}
  The symbol $\Omega$, and the conditions (a1), (a2), (a3), (3.1),
  (3.2) in this proof refer to \cite{HuKaMa}.

  In the notation of  \cite{HuKaMa}, \eqref{eq:w:macro} reads
  \begin{displaymath}
    a_{00} \frac{\pp^2 U}{\pp T^2} = \sum_{\alpha,\beta = 1}^d a_{\alpha\beta}
   \frac{\pp^2 U}{\pp X_\alpha \pp X_\beta} + b,
  \end{displaymath}
  where $a_{00} = I, a_{\alpha\beta} = (\bbC_{i\alpha}^{j\beta}(\D
  U))_{i,j = 1}^d$, where $\bbC(\mF) = \D_\mF^2 W(\mF)$, $b_i =
  \sum_{\alpha, \beta, j = 1}^d \frac{\pp \bbC_{i\alpha}^{j\beta}}{\pp
    X_\alpha} \frac{\pp U_j}{\pp X_\beta}$, and $a_{0,i} = a_{i,0} =
  0$, for $i = 1, \dots, d$. Condition (a1) is trivially satisfied and
  condition (a2) follows from the fact that $W \in C^2$. Condition
  (a3) is the Legendre--Hadamard condition for $\bbC(\mF)$, which we
  know from Proposition \ref{th:stab_small} to be satisfied for $|\mF|
  \leq \kappa$. Hence, choosing $\Omega = \R^d \times \R^d \times
  B_{\R^{d\times d}}(\mO, \kappa)$ we obtain (a3).

  Condition (3.1) is satisfied for $s = 3$. Condition (3.2) is
  satisfied since we assumed that $\| \D_X U_0 \|_{L^\infty} <
  \inv$. The requirement that $s \geq d/2+1$ holds since we
  have restricted $d \leq 3$. This ensures the existence of a solution
  with the stated regularity.

  Since $U^\c \in C^1([0, T^\c]; W^{3,2}) \subset C^1([0, T^\c];
  W^{1,\infty})$ and since $\|\D U^\c(0)\|_{L^\infty} < \kappa$ it
  follows that $\| \D U^\c(T) \|_{L^\infty} \leq \kappa$ for
  sufficiently short time. Thus choosing $T^\c$ sufficiently small, we
  obtain that $\max_{T \in [0, T^\c]} \| \D U^\c(T)\|_{L^\infty} <
  \inv$.
\end{proof}

Upon reverting the scaling \eqref{eq:w:cb_scale}, we obtain the
existence of a trajectory $u^\c(x, t) := \eps^{-1}U^\c(\eps x, \eps
t)$, defined for $x \in \R^d, t \in [0, t^\c]$, where $t^\c := T^\c
/ \eps$, satisfying the following conditions:
\begin{align}
  \label{eq:w:cb}
  &(\ddot u^\c, u)_\Om + \< \del\Ec(u^\c), u\> = 0 \qquad \forall\, u
  \in W^{1,2}, \quad t \in (0, t^\c], \\
  \label{eq:w:micro_init}
  &u^\c(x,0) = \eps^{-1} U_0^\c(\eps x),\quad \dot{u}^\c(x, 0) =
  U_1^\c(\eps x), \qquad \text{and}\\
  \label{eq:w:cb_reg}
  & \| \D_x^m \D_t^l u^\c \|_{L^\infty([0, t^\c], L^2)} \leq
  C_{m,l} \eps^{m+l-1-d/2} \quad \text{for } m, l \in \N, 1 \leq m+l
  \leq 4,
\end{align}
where $C_{m,l} = \| \D_X^m \D_T^l U^\c \|_{L^\infty([0, T^\c], L^2)}$.

We also imposed in Proposition \ref{th:HuKaMa} that $T^\c$ is chosen
sufficiently small to ensure that $U^\c(T) \in \Bc$ for all $T \in [0,
T^\c]$. This implies that $u^\c(t) \in \Bc$ for all $t \in [0, t^\c]$,
and hence we may conclude that
\begin{displaymath}
  \< \ddel\Ea(u^\c(t)) v, v \> \geq \smfrac12\gamma \| \D \barv \|_{L^2}^2
  \qquad \forall\, v \in \ell^2, \quad t \in [0, t^\c_1].
\end{displaymath}
\subsection{Main result}
After the preparation we can state our result on the convergence of solutions
of~(\ref{eq:w:at}).

\begin{theorem}
  \label{th:w:mainthm}
  Suppose that $d \in \{1,2,3\}, k \geq 4$, and $\mA\cdot\L$ is
  stable. Let $\| \D U_0^\c \|_{L^\infty} < \inv$, where $\inv$ is
  chosen sufficiently small so that Proposition \ref{th:stab_small}
  applies. Finally, suppose that $\Md^{(j,2)}$ is finite for $j = 2,
  3, 4$, and let $C_{\rm init} > 0$ be a fixed constant.

  Then there exists $T^\a = \eps^{-1} t^\a \in (0, T^\c]$ and
  $\eps_0 > 0$ such that, for all $\eps \leq \eps_0$, and for any
  initial data $u^\a_0, u^\a_1$ with
  \begin{equation}
    \label{eq:errbnd_initial_data}
    \eps^{d/2} \b\| \D \Itil u^\a_0 - \D u^\c(0) \b\|_{L^2} +
    \eps^{d/2}  \b\| \D \Itil
    u^\a_1 - \D \dot{u}^\c(0) \b\|_{L^2} \leq C_{\rm init} \eps^{2},
  \end{equation}
  there exists $u^\a \in C^2([0, t^\a]; \ell^2)$ satisfying
  \eqref{eq:w:at}, and
  \begin{equation}
    \label{eq:w:errest}
    \max_{0 \leq t \leq t^\a} \eps^{d/2} \b( \| \D \Itil u^\a(t) - \D
    u^\c(t) \|_{L^2} + \| \Itil \dot{u}^\a(t) - \dot{u}^\c(t) \|_{L^2}
    \b) \leq C \eps^{2},
  \end{equation}
  where $C = C(\Lam, (\Md^{(j,2)})_{j = 2}^4, U^\c, C_{\rm init})$.

  If $d \leq 2$ or $k \geq 5$, then we may choose $t^\a = t^\c$.
\end{theorem}
\begin{proof}
  We give here an outline of the proof, but will establish the key
  technical results in the following subsections. All constants in
  this proof may depend on any property of $U^\c$.

  {\it 1. Setup: } From Lemma \ref{th:w:at_local_ex} we obtain local
  existence for the atomistic problem: there exists $t_1 > 0$ and $u
  \in C^2([0, t_1]; \ell^2)$ satisfying \eqref{eq:w:at}.  Let $z(t) :=
  \tilu^\c(t)|_\L = (\zz \ast u^\c(t))|_\L$ and let
  \begin{displaymath}
    e(\xi, t) := u^\a(\xi, t) - z(\xi,t),
  \end{displaymath}
  whenever the left-hand side is well-defined. Moreover, let $w \in
  \ell^2$ such that $\tilw(\xi,t) = e(\xi,t)$.

  It follows from \eqref{eq:w:cb_reg}, Lemma
  \ref{th:convolution_error}, and Lemma \ref{th:interp_error}, that
  the initial error satisfies
  \begin{equation}
    \label{eq:w:err_init}
    \| \dot{e}(0) \|_{\ell^2}^2 + \smfrac\gamma2 \|\D\bare(0) \|_{L^2}^2 \leq C_0 \eps^{4 - d},
  \end{equation}
  where $C_0$ depends on $U_0^\c, U_1^\c$, $C_{\rm init}$ but is
  independent of $\eps$. We fix a constant $C_* > C_0$, which will be
  specified later on. Since $u^\a \in C^2([0, t_1], \ell^2)$ (and
  hence $e \in C^2([0, t_1], \ell^2)$) upon choosing $t_1$
  sufficiently small, we obtain
  \begin{equation}
    \label{eq:w:err_traj_t1}
    \| \dot{e}(t) \|_{\ell^2}^2 + \smfrac\gamma2 \|\D\bare(t)
    \|_{L^2}^2 \leq C_* \eps^{4 - d} \qquad \forall t \in [0, t_1].
  \end{equation}
  Thus the task is to prove that we may choose $t_1 \gtrsim
  \eps^{-1}$.

  From the definition of $z$ and the assumption that $\| \D u^\c(t)
  \|_{L^\infty} \leq \inv' < \inv$ for all $t \in [0, t^\c]$, we
  deduce that $|D_\rho z(t,\xi)| \leq \inv'$ for all $t, \xi, \rho$
  (cf. the proof of Lemma \ref{th:stress:moderr_negSob}). Choosing
  $\eps_0$ sufficiently small, then for $\eps \leq \eps_0$ we obtain
  from \eqref{eq:w:err_traj_t1} that
  \begin{equation}
    \label{eq:w:Ea_smooth_t1}
    z(t) + \theta e(t) \in \Ba \qquad \forall t \in [0, t_1], \quad \theta
    \in [0, 1].
  \end{equation}
  Since we assumed that $\kappa$ was chosen sufficiently small,
  Proposition \ref{th:stab_small} implies that
  \begin{equation}
    \label{eq:w:stab_traj_t1}
    \b\< \ddel\Ea(z(t)+\theta e(t)) v, v \b\> \geq \smfrac\Lam2 \| \D\barv
    \|_{L^2}^2 \qquad \forall v \in \ell^2, \quad t \in [0, t_1],
    \quad \theta \in [0, 1].
  \end{equation}

  {\it 2. Error equation: } Testing \eqref{eq:w:at} with $\dot{e}$ and
  \eqref{eq:w:cb} with $\dot{w}$, yields
  \begin{equation}
    \label{eq:w:err_eqn_1}
    (\ddot{e}, \dot{e})_\L + \< \del\Ea(u^\a) - \del\Ea(z),
    \dot{e} \> = \B\{ (\ddot{u}^\c, \dot w)_\Om - (\ddot{z},
    \dot{e})_\L \B\} + \B\{ \<\del\Ec(u^\c), \dot{w} \> - \<\del\Ea(z), \dot{e} \> \B\}.
  \end{equation}
  Since $\tilw = e$, the first group on the right-hand side can be
  rewritten as (cf. \eqref{eq:fa:1})
  \begin{displaymath}
    \B\{ (\ddot{u}^\c, \dot w)_\Om - (\ddot{z},
    \dot{e})_\L \B\} = (\ddot{u}^\c - \ddot{z}, \dot{w})_\Om
    =: (\alpha, \dot{w})_\Om,
  \end{displaymath}
  where $z$ and $w$ are identified with their first-order
  interpolants.  Similarly, using Proposition~\ref{th:Sa}, we can
  write
  \begin{displaymath}
    \B\{ \<\del\Ec(u^\c), \dot{w} \> - \<\del\Ea(z), \dot{e} \> \B\} = (\Scb(u^\c) - \Sa(z), \D\dot{w})_\Om =: (\beta, \D\dot{w})_\Om,
  \end{displaymath}
  Moreover, according to \eqref{eq:w:Ea_smooth_t1} we can expand
  \begin{displaymath}
    \b\< \del\Ea(u^\a) - \del\Ea(z),
    \dot{e} \b\> =: \int_0^1 \b\< \ddel\Ea(z + \theta e) e, \dot{e}
    \b\> \dd\theta
    =: \< H e, \dot{e} \>,
  \end{displaymath}
  to rewrite \eqref{eq:w:err_eqn_1} as
  \begin{equation}
    \label{eq:w:err_eqn_2}
    \frac{\dd}{\dt} \B\{\| \dot{e} \|_\L^2 + \< H
    e, e \> \B\} = \< \dot{H} e, e\> + 2 (\alpha, \dot{e})_\Om +
    2 (\beta, \D\dot{w})_\Om.
  \end{equation}
  We define $E^2(t) := \| \dot{e} \|_{\ell^2}^2 + \< H e, e \>$, and
  note that \eqref{eq:w:stab_traj_t1} immediately implies that
  \begin{equation}
    \label{eq:w:E2_positive}
    \| \dot{e} \|_{\ell^2}^2
    + \smfrac\Lam2 \| \D \bare \|_{L^2}^2  \leq E^2(t) \leq C_H \b( \| \dot{e}
    \|_{\ell^2}^2 + \smfrac{\Lam}{2} \| \D\bare \|_{L^2}^2\b) \qquad \forall t \in [0, t_1],
  \end{equation}
  where $C_H$ depends only on $M^{(2)}$.
  We integrate \eqref{eq:w:err_eqn_2} over $(0, s)$, $s \leq t_1$, to
  obtain
  \begin{equation}
    \label{eq:w:err_eqn_3}
    E^2(s) = E^2(0) + \int_0^s \B\{ \< \dot{H} e, e\> + 2(\alpha,
    \dot{w})_\Om + 2(\beta, \dot{w})_\Om \B\} \dt,
  \end{equation}
  and proceed to estimate the three terms on the right-hand side
  separately.

  {\it 3. Consistency estimates and nonlinearity: } The second and
  third term on the right-hand side of \eqref{eq:w:err_eqn_3} measure
  the consistency between the atomistic and Cauchy--Born
  model. Standard interpolation error results (cf. \S\ref{sec:w:cons}
  for the details) yield the estimate $\| \alpha \|_{L^2} \lesssim \|
  \D^2 \ddot{u}^\c \|_{L^2} \lesssim \eps^{3-d/2}$, and hence,
  applying Cauchy's inequality and \eqref{eq:interp:norm_equiv_Lp}, we
  obtain
  \begin{equation}
    \label{eq:w:prf_mom_cons}
    \int_0^s 2(\alpha,
    \dot{w})_\Om \dt \leq C_\alpha \bg( t_1 \eps^{5-d} + \eps \int_0^s
    E^2 \dt \bg),
  \end{equation}
  where $C_\alpha$ depends only on the trajectory $U^\c$ through the
  bounds \eqref{eq:w:cb_reg}, but it is independent of $\eps$.

  The third term on the right-hand side of \eqref{eq:w:err_eqn_3}
  requires an estimate that it similar to Lemma
  \ref{th:stress:moderr_negSob}. We establish the required variant in
  \S\ref{sec:w:cons}:
  \begin{equation}
    \label{eq:w:prf_frc_cons_pre}
    (\beta, \D\dot w)_{L^2} \leq C \eps^{3-d/2} \| \dot{w} \|_{L^2},
  \end{equation}
  where $C = C(U^\c, (\Md^{(j,2)})_{j=2}^4)$.
  Integrating over $(0, s)$, $s \leq t_1$, and applying Cauchy's
  inequality, we again obtain
  \begin{equation}
    \label{eq:w:prf_frc_cons}
    \int_0^s 2(\beta, \D \dot{w})_{L^2} \dt \leq C_\beta \bg( t_1
    \eps^{5-d} + \eps \int_0^s E^2 \dt \bg),
  \end{equation}
  where $C_\beta$ depends on $(\Ms^{(j,2)})_{j = 2}^4$ and on $U^\c$, but
  not on $\eps$ (provided $\eps \leq \eps_0$, which we chose above so
  that $u^\a = z + e \in \Ba$).

  Finally, the first term on the right-hand of \eqref{eq:w:err_eqn_3}
  side is not a consistency term, but must be otherwise
  controlled. Our argument is a refinement of a method due to
  Makridakis \cite{Makr:wave} for the numerical analysis of nonlinear
  wave equations.
  In \S\ref{sec:w:nonlin} we show that an integrating by parts argument leads to
  \begin{equation}
    \label{eq:w:prf_nonlin}
    \int_0^s \b\< \dot{H} e, e \b\> \dt \leq
    C_{\rm nl} \bg( E^3(s) + E^3(0) + \int_0^s \B( \eps E^2 + \eps E^3
    +E^4\B) \dt \bg),
  \end{equation}
  where $C_{\rm nl} = C_{\rm nl}(M^{(3)}, M^{(4)}, U^\c)$.

  {\it 4. Gronwall lemma: } Combining \eqref{eq:w:prf_mom_cons},
  \eqref{eq:w:prf_frc_cons}, \eqref{eq:w:prf_nonlin} with
  \eqref{eq:w:err_eqn_3} yields
  \begin{align*}
    E^2(s) \leq\,& E^2(0) + (C_\alpha+C_\beta) \bg( t_1 \eps^{5-d} +
    \eps
    \int_{0}^s E^2 \dt \bg) \\
    & + C_{\rm nl} \bg( E^3(0) + E^3(s) + \int_0^s \B(\eps E^2 + \eps
    E^3 + E^4\B) \dt \bg).
  \end{align*}
  We employ \eqref{eq:w:err_init}, \eqref{eq:w:err_traj_t1},
  \eqref{eq:w:E2_positive} and $T_1 := t_1 \eps$, to obtain
  \begin{align*}
    E^2(s) \leq\,& \b[ C_HC_0 + (C_\alpha+C_\beta) T_1 \b] \eps^{4-d}
    + \eps (C_\alpha+C_\beta+C_{\rm nl}) \int_0^s E^2 \dt \\
    & + C_{\rm nl} \bg( C_H^{3/2} C_0^{3/2} \eps^{6-3d/2} + C_H^{3/2}
    C_*^{3/2} \eps^{6-3d/2} + \int_0^s C_H C_* \eps^{4-d} E^2 \dt \bg).
  \end{align*}
  Since $\eps^{6-2d/2} \leq \eps^{4-d} \cdot \eps_0^{2-d/2}$ and since
  $d \leq 3$, choosing $\eps_0$ sufficiently small ensures that
  \begin{displaymath}
    C_H^{3/2} C_0^{3/2} \eps^{6-3d/2} + C_H^{3/2}
    C_*^{3/2} \eps^{6-3d/2} \leq C_H C_0 \eps^{4-d}.
  \end{displaymath}
  (However, if $d = 3$ then $\eps^{4-d} = \eps$, hence the integral
  term cannot be made arbitrarily small for $s = O(1/\eps)$.) Upon
  defining $C_1 = C_1(T_1) = 2C_HC_0 + (C_\alpha+C_\beta) T_1$ and
  $C_2 = C_2(C_*, \eps_0) = C_\alpha+C_\beta+C_{\rm nl}(1+C_H C_*
  \eps_0^{3-d})$, yields
  \begin{equation}
    \label{eq:w:err_eqn_4simple}
    E^2(s) \leq C_1 \eps^{4-d} + C_2 \int_0^s E^2 \dt, \qquad
    \text{for } 0 \leq s \leq t_1.
  \end{equation}
  Applying Gronwall's inequality, we obtain
  \begin{equation}
    \label{eq:w:err_2}
    \max_{0 \leq t \leq t_1} E^2(t) \leq C_1 \exp(C_2 T_1) \eps^{4-d}.
  \end{equation}
  We observe that $C_1 \exp(C_2 T_1) \to 2 C_H C_0$ as $T_1 \to 0$.

  We now choose $C_*, \eps_0, T^\a$ in the following order: 1. $C_* :=
  4 C_H C_0$ (or any constant larger than $2 C_H C_0$); 2. $\eps_0
  \leq 1$
  and sufficiently small as required in the estimates up to this point
  (dependent in particular on $C_*$); and 3. $T^\a > 0$ sufficiently
  small so that
  \begin{displaymath}
     C_1(T^\a) \exp(C_2(C_*, 1) T^\a) =: C_3(T^\a) \leq C_* \qquad
     \text{and} \qquad T^\a \leq T^\c.
  \end{displaymath}
  This is possible due to the fact that $C_3(T^\a) \to \smfrac12 C_*$ as $T^\a
  \to 0$.  We set $t^\a := T^\a/\eps$.

  Using these definitions, we can estimate \eqref{eq:w:err_2} above by
  \begin{equation}
    \label{eq:w:err_3}
    \max_{0 \leq t \leq t_1} E^2(t) \leq C_3(T_1) \eps^{4-d} \leq C_*
    \eps^{4-d}.
  \end{equation}

  {\it 5. Continuation argument: } Suppose that $t_1 \leq t^\a$ is
  chosen maximally, such that the trajectory $u^\a$ exists in $[0,
  t_1]$ and \eqref{eq:w:err_traj_t1} holds with the choice of $C_*$ we
  made above. If $t_1 < t^\a$, then we end up with the stronger error
  estimate \eqref{eq:w:err_3}. We apply Lemma \ref{th:w:at_local_ex}
  again to extend $u^\a$ to an interval $[t_1, t_2]$ where $t_1 < t_2
  \leq t^\a$. Since \eqref{eq:w:err_3} holds in $[0, t_1]$, and since
  $C_3(T_1) < C_*$ when $T_1 < T^\a$, choosing $t_2$ sufficiently
  close to $t_1$ implies that \eqref{eq:w:err_traj_t1} holds in $[0,
  t_2]$. This contradicts the maximality of $t_1$ and hence we must
  have $t_1 = t^\a$. Thus, we conclude that \eqref{eq:w:err_traj_t1}
  holds in $[0, t^\a]$, where $t^\a = \eps^{-1} T^\a$ and $T^\a > 0$.

  {\it 6. Maximal time interval for $d < 3$: } If $d < 3$ then one
  varies the argument following \eqref{eq:w:err_2}, exploiting the
  fact that $\eps_0^{3-d} \to 0$ as $\eps_0 \to 0$, as
  follows. 1. Choose $C_* := C_1(T^\c) \exp( C_2(1, 1)
  T^\c)$. 2. Choose $\eps_0$ sufficiently small as required by the
  previous estimates, and in addition sufficiently small so that $C_*
  \eps_0 \leq 1$. 3. Choose $T^\a = T^\c$. The rest of the argument is
  analogous.

  {\it 7. Maximal time interval for $k \geq 5$: } Now suppose that $d
  = 3$ and $k \geq 5$, then we need to vary the proof starting from
  \eqref{eq:w:prf_nonlin}. The assumption that $k \geq 5$ allows us to
  integrate the term $\< \dot{H} e, e\>$ by parts twice, to prove in
  \S\ref{sec:w:nonlin} that
  \begin{equation}
    \label{eq:w:prf_nonlin_5}
    \begin{split}
      \int_0^s \b\< \dot{H} e, e \b\> \dt \leq
      C_{\rm nl} \bg(& E^3(s) + E^3(0) + E^4(s) + E^4(0) \\
      & + \int_0^s \B( \eps E^2 + \eps E^3 + \eps E^4 + E^5(t) \B) \dt
      \bg),
    \end{split}
  \end{equation}
  where $C_{\rm nl}$ depends only on $(M^{(j)})_{j = 3}^5$. Exploiting
  the higher powers of $E$ and arguing similarly as in {\it 4}, upon
  choosing $\eps_0$ sufficiently small, we obtain
  \begin{displaymath}
    \int_0^s \b\< \dot{H} e, e \b\> \dt \leq C_H C_0 \eps^{4-d} +
    C_{\rm nl} \eps \int_0^s E^2 \dt,
  \end{displaymath}
  and eventually
  \begin{displaymath}
    \max_{0 \leq t \leq t_1} E^2(t) \leq C_1 \exp(C_2 T_1) \eps^{4-d},
  \end{displaymath}
  where $C_1 = 2 C_H C_0 + (C_\alpha+C_\beta) T_1$ and $C_2 =
  C_\alpha+C_\beta + C_{\rm nl}$. Since $C_1$ and $C_2$ are
  independent of $C_*$ we can now argue as in step {\it 6} to deduce that
  we may choose $T^\a = T^\c$.
\end{proof}

\begin{lemma}
  \label{th:w:at_local_ex}
  Let $u_0, u_1 \in \ell^2$, $\| \D u_0 \|_{L^\infty} < \inv$, then
  there exists $t_1 > 0$ and a unique trajectory $u \in C^2([0, t_1];
  \ell^2)$ satisfying \eqref{eq:w:at} with $u(0) = u_0, \dot{u}(0)
  = u_1$.
\end{lemma}
\begin{proof}
  The estimate \eqref{eq:bound_delj} with $j = 2$ implies the
  existence of $\delta > 0$ such that, for all $v_1, v_2 \in
  B_{\ell^2}(u_0, \delta)$,
  \begin{displaymath}
    \b|\< \del\Ea(v_1) - \del\Ea(v_2), w \>\b| \leq M^{(2)} \|
    \D \barv_1 - \D\barv_2 \|_{L^2}\| \D\barw\|_{L^2} \lesssim M^{(2)} \| v_1
    - v_2 \|_{\ell^2} \| w \|_{\ell^2}.
  \end{displaymath}
  that is, $\del\Ea : \ell^2 \to \ell^2$ is Lipschitz. Hence the
  result follows from Picard's theorem.
\end{proof}

\subsection{Consistency estimates}
\label{sec:w:cons}
In this section, we prove the two consistency estimates
\eqref{eq:w:prf_mom_cons} and \eqref{eq:w:prf_frc_cons}. We begin by
considering \eqref{eq:w:prf_mom_cons}, which follows from standard
interpolation error estimates.

\begin{lemma}
  \label{th:w:momentum_cons}
  Let $v \in W^{2,2}$, and let $\ol{\zz\ast v}$ be the first-order
  interpolant of $\zz\ast v|_\L$, then
  \begin{displaymath}
    \b\| v - \ol{\zz\ast v} \b\|_{L^2} \lesssim \| \D^2 v \|_{L^2}.
  \end{displaymath}
\end{lemma}
\begin{proof}
  We first apply a triangle inequality,
  \begin{displaymath}
    \b\| v - \ol{\zz\ast v} \b\|_{L^2} \leq \b\| v - \zz \ast v
    \b\|_{L^2} + \b\|\zz \ast v - \ol{\zz\ast v} \b\|_{L^2}.
  \end{displaymath}
  The first term on the right-hand side is estimated using Lemma
  \ref{th:convolution_error},
  \begin{displaymath}
    \b\| v - \zz \ast v
    \b\|_{L^2} \lesssim \| \D^2 v \|_{L^2}.
  \end{displaymath}
  The second term on the right-hand side is a standard
  Q1-interpolation error, which can be estimated by
  \cite{Ciarlet:1978}
  \begin{displaymath}
    \b\|\zz \ast v - \ol{\zz\ast v} \b\|_{L^2} \lesssim \| \D^2 (\zz
    \ast v) \|_{L^2} = \| \zz \ast \D^2 v \|_{L^2} \leq \| \D^2
    v\|_{L^2}. \qedhere
  \end{displaymath}
\end{proof}

\begin{proof}[Proof of \eqref{eq:w:prf_mom_cons}]
  Applying Lemma \ref{th:w:momentum_cons}, the regularity bound
  \eqref{eq:w:cb_reg}, and the norm-equivalence $\| \dot{w} \|_{L^2}
  \approx \| \dot{e} \|_{\ell^2} \leq E$ (see \S~\ref{sec:interp}), we
  can estimate
  \begin{align*}
    \int_0^s (\alpha, \dot{w})_{L^2} \dt \lesssim\,& \int_0^s \| \D^2
    \ddot{u}^\c \|_{L^2} \| \dot{w} \|_{L^2} \dt
    \lesssim \int_0^s \eps^{3-d/2} E \dt.
  \end{align*}
  Applying Cauchy's inequality to $\eps^{3-d/2} E = \eps^{(5-d)/2} \cdot (\eps^{1/2} E)$, we
  obtain
  \begin{displaymath}
    \int_0^s (\alpha, \dot{w})_{L^2} \dt \lesssim \int_0^s \B(
    \eps^{5-d} + \eps E^2 \B) \dt \leq t_1 \eps^{5-d} + \eps \int_0^s
    E^2 \dt \qedhere
  \end{displaymath}
\end{proof}

We now turn towards \eqref{eq:w:prf_frc_cons}, which is the crucial
ingredient in the dynamic analysis.  Our strategy will be to integrate
by parts, and estimate $\D\beta$, which is the error in the divergence
of the stresses. To that end, we first need to prove that $x \mapsto
\Sa(u; x)$ is differentiable. For $\Scb$, the analogous result follows
from Lemma \ref{th:defn_W}.

\begin{lemma}
  Let $u \in \Ba$, then $x \mapsto \Sa(u; x) \in
  \WW^{1,\infty}$, with
  \begin{equation}
    \label{eq:w:divSa}
    {\rm div}\, \Sa(u; x) = \sum_{\rho\in\Rg} \sum_{\xi\in\L}
    \Es_{\xi,\rho}(u) \D_\rho \ww_{\xi,\rho}(x),
  \end{equation}
  where $\D_\rho \ww_{\xi,\rho}(x) = \D\ww_{\xi,\rho}(x) \cdot \rho$.
\end{lemma}
\begin{proof}
  Since $\ww_{\xi,\rho} \geq 0$ and $(\ww_{\xi,\rho})_{\xi\in\L}$ is a
  partition of unity, it is straightforward to show that $|\Sa(u;x)|
  \leq M^{(1)}$ for all $x \in \Om$.

  The rest of the proof is devoted to the assertion that $\Sa$ is
  Lipschitz continuous. First, we compute the Lipschitz constant of
  $\ww_{\xi,\rho}$. Let $L$ be the global Lipschitz constant of $\zz$,
  let $S_x := {\rm supp}(\zz(\;\cdot\; - x))$, and let $J_{x,y} := \{
  s \in (0, |\rho|) \sep \xi+s\frac{\rho}{|\rho|} \in S_x \cap S_y\}$;
  then
  \begin{align*}
    \b|\ww_{\xi,\rho}(x) - \ww_{\xi,\rho}(y)\b| \leq\,& \frac{1}{|\rho|}
    \int_{I_{x,y}} \b| \zz\b(\xi+s\smfrac{\rho}{|\rho|} -
    x\b) - \zz\b(\xi+s\smfrac{\rho}{|\rho|} -
    y\b) \b| \ds \\
    \leq\,& \frac{1}{|\rho|}
    \int_{I_{x,y}} L |x - y| \ds \leq \frac{2L {\rm
        diam}(S_x)}{|\rho|} \, |x - y| \lesssim |\rho|\,|x-y|.
  \end{align*}

  Next, we note that $\sum_{\xi \in\L} (\ww_{\xi,\rho}(x) -
  \ww_{\xi,\rho}(y) ) = 0$, and hence we can rewrite
  \begin{align*}
    \Sa(u; x) - \Sa(u; y) =\,&
    \sum_{\rho\in\Rg} \sum_{\xi\in\L} [\Es_{\xi,\rho}(u) \otimes \rho]
    \b(\ww_{\xi,\rho}(x) - \ww_{\xi,\rho}(y) \b) \\
    =\,&
    \sum_{\rho\in\Rg} \sum_{\xi\in\L} \b[(\Es_{\xi,\rho}(u) - \Es_{\xi,\rho}(0)) \otimes \rho\b]
    \b(\ww_{\xi,\rho}(x) - \ww_{\xi,\rho}(y) \b).
  \end{align*}
  Expanding $\Es_{\xi,\rho}$, using that fact that
  \begin{displaymath}
    \# \{ \xi \in \L \sep \ww_{\xi,\rho}(x) \neq 0 \text{ and }
    \ww_{\xi,\rho}(y) \neq 0 \} \lesssim |\rho|,
  \end{displaymath}
  and the Lipschitz estimate for $\ww_{\xi,\rho}$, we obtain
  \begin{displaymath}
    \b|\Sa(u; x) - \Sa(u; y)\b| \lesssim M^{(2)} \| \D u \|_{L^\infty} |x
    - y|.
  \end{displaymath}

  The formula for the divergence \eqref{eq:w:divSa} is easy to establish.
\end{proof}

We see from the formula for ${\rm div} \Sa$ that we will require
symmetries of the weighting functions $\D_\rho \ww_{\xi,\rho}$, which
we establish next.

\begin{lemma}
  Let $x \in \R^d$ and $\rho \in \Rg$, then
  \begin{align}
    \label{eq:dS_symm_0}
    \sum_{\xi\in\L} \D_\rho \ww_{\xi,\rho}(x) =\,& 0, \\
    \label{eq:dS_symm_1}
    \sum_{\xi\in\L} \D_\rho \ww_{\xi,\rho}(x) \, (\xi-x) =\,&\rho,
    \quad \text{and} \\
    \label{eq:dS_symm_2}
    \sum_{\xi \in \L} \D_\rho \ww_{\xi,\rho}(x) \, (\xi-x) \otimes
    (\xi-x)  =\,& - \rho \otimes \rho.
  \end{align}
\end{lemma}
\begin{proof}
  First, we compute $\D_\rho \ww_{\xi,\rho}(x)$:
  \begin{align}
    \notag
    \D_\rho \ww_{\xi,\rho}(x) =\,& - \int_0^1 \D\zz(\xi+t\rho-x) \rho
    \dt \\
    \label{eq:dS:Drhoww}
    =\,& - \int_0^1 \smfrac{\dd}{\dt} \zz(\xi+t\rho-x) \dt =
    \zz(\xi-x) - \zz(\xi+\rho-x).
  \end{align}
  Using this identity and the fact that the nodal interpolant
  reproduces affine functions, we immediately obtain
  \eqref{eq:dS_symm_0} and \eqref{eq:dS_symm_1}

  To prove \eqref{eq:dS_symm_2}, consider
  \begin{align*}
    & \sum_{\xi \in \Z^d} \D_\rho \ww_{\xi,\rho}(x) \, (\xi-x) \otimes
    (\xi-x) \\
    =\,& \sum_{\xi \in \Z^d} \b[\zz(\xi-x) - \zz(\xi+\rho-x)\b]\, (\xi-x) \otimes
    (\xi-x) \\
    =\,& \sum_{\xi \in \Z^d}\zz(\xi-x) \b[ (\xi-x) \otimes
    (\xi-x) - (\xi-\rho-x) \otimes (\xi-\rho-x) \b] \\
    =\,&  \sum_{\xi \in \Z^d}\zz(\xi-x) \b[
    \rho \otimes (\xi-x) + (\xi-x) \otimes \rho - \rho\otimes \rho \b].
  \end{align*}
  Using again the fact that the nodal interpolant reproduces affine
  functions, we obtain
  \begin{displaymath}
     \sum_{\xi \in \Z^d} \D_\rho \ww_{\xi,\rho}(x) \, (\xi-x) \otimes
    (\xi-x) = \rho \otimes (x - x) + (x-x) \otimes \rho - \rho \otimes
    \rho = - \rho \otimes \rho. \qedhere
  \end{displaymath}
\end{proof}

The key step towards proving \eqref{eq:w:prf_frc_cons} is the next
result, which is analogous to Lemma \ref{th:stress:c2}.

\begin{lemma}
  Let $u \in \WW^{4,\infty}_\loc \cap \Bc$; then
  \begin{align*}
    \b| {\rm div} \Sa(u; x) - {\rm div} \Scb(u; x) \b| \lesssim\,&
    \sum_{\bfrho\in\Rg^4} \md^{(\infty)}(\bfrho) \prod_{j = 2}^4 \| \D^2 u
    \|_{L^\infty(x+\nu_{\rho_1\rho_j})}  \\
    & + \sum_{\bfrho\in\Rg^3} \md^{(\infty)}(\bfrho)  \| \D^3 u
    \|_{L^\infty(x+\nu_{\rho_1,\rho_2})} \| \D^2 u
    \|_{L^\infty(x+\nu_{\rho_1,\rho_3})} \\
    & + \sum_{\bfrho\in\Rg^2}
    \md^{(\infty)}(\bfrho) \| \D^4 u \|_{L^\infty(x+\nu_{\rho_1,\rho_2})}.
  \end{align*}
  where $\md^{(\infty)}$ is defined in \S\ref{sec:intro:V:decya}.
\end{lemma}
\begin{proof}
  Recall from the proof of Lemma \ref{th:stress:c2} the definition of
  the sets $\nu_{\xi,\vsig}$ and $\nu_{\rho,\vsig}$ and the quantities
  $\eps_{\xi,\vsig}$ and $\delta_{\xi,\vsig}$. In addition, we define
  $\mu_{\xi,\vsig} := \| \D^4 u \|_{L^\infty(\nu_{\xi,\vsig})}$.

  Recall from \eqref{eq:w:divSa} the algebraic expression for ${\rm
    div} \Sa$.  Since $\Scb = \sum_{\rho\in\Rg} V_\rho \otimes
  \rho$, where $V_\rho := V_\rho(\D u(x) \cdot \Rg)$, we also obtain
  \begin{equation}
    \label{eq:dS_divSc}
    {\rm div} \Scb(x) = \sum_{\rho,\vsig \in \Rg} V_{\rho\vsig}
    [\;\cdot\;, \D_\rho\D_\vsig u(x)].
  \end{equation}

  To proceed, we expand $\Es_{\xi,\rho}$ in \eqref{eq:w:divSa} to
  third order:
  \begin{align}
    \notag
    \Es_{\xi,\rho} =\,& V_\rho + \B\{\sum_{\vsig\in\Rg}
    V_{\rho\vsig} [\;\cdot\;, D_\vsig u(\xi)-\D_\vsig u(x)]\B\} \\
    \notag
    & + \B\{\sum_{\vsig,\tau\in\Rg} V_{\rho\vsig\tau}[\;\cdot\;,
    D_\vsig u(\xi) - \D_\vsig u(x), D_\tau u(\xi) - \D_\tau u(x)]\B\} +
    E_2(\xi,\rho), \\
    \label{eq:dS:10}
    =:\,& V_\rho + T_1(\xi,\rho) + E_1(\xi,\rho) + E_2(\xi,\rho) \\
    \notag
    \text{where}\quad |E_2(\xi,\rho)| \lesssim\,&
    \sum_{\rho_2,\rho_3,\rho_4 \in \Rg} \smfrac{m(\rho,
    \rho_2,\rho_3,\rho_4) }{|\rho|}
    \prod_{i=2}^4 (|\rho|+|\rho_i|) \, \eps_{\xi,\rho_i},
  \end{align}
  where the estimate for $E_2$ is fully analogous to the estimate of
  $E_1$ in \eqref{eq:stress:c2_8} (taking into account the finite
  support of $\ww_{\xi,\rho}$). Summing over $\xi\in\L, \rho\in\Rg$,
  employing the fact that $\sum_{\xi\in\L} |\D_\rho \ww_{\xi,\rho}|
  \leq 2$ (this is an immediate consequence of \eqref{eq:dS:Drhoww}),
  and otherwise arguing as in the proof of Theorem \ref{th:stress:c2},
  we obtain
  \begin{align}
    \notag
    |E_2'| :=\,& \sum_{\rho\in\Rg}\sum_{\xi\in\L} |\D_\rho \ww_{\xi,\rho}|
    |E_2(\xi,\rho)|
    \lesssim \sum_{\bfrho\in\Rg^4} \frac{m(\bfrho) |\bfrho|^3}{\prod_{i = 1}^4 |\rho_i|}
    \sum_{\xi\in\L} |\D_{\rho_1} \ww_{\xi,\rho_1}| \prod_{i = 2}^4 |\rho_i|
    \eps_{\xi,\rho_i} \\
    \notag
    \lesssim\,& \sum_{\bfrho\in\Rg^4} \frac{m(\bfrho)
    |\bfrho|^3}{|\rho_1|} \prod_{i = 2}^4
    \max_{\substack{\xi\in\L \\ \ww_{\xi,\rho_1}(x) \neq 0}}
    \eps_{\xi,\rho_i} \, \sum_{\xi\in\L} |\D_{\rho_1} \ww_{\xi,\rho_1}|\\
    \label{eq:dS:E2'}
    \lesssim\,& \sum_{\bfrho\in\Rg^4} \md^{(\infty)}(\bfrho)
    \prod_{i = 2}^4 \| \D^2 u \|_{L^\infty(x+\nu_{\rho_1,\rho_i})}.
  \end{align}

  Moreover, using \eqref{eq:dS_symm_0}, it is straightforward to treat
  the first term in \eqref{eq:dS:10}:
  \begin{equation}
    \label{eq:dS:15}
    \sum_{\rho\in\Rg} \sum_{\xi \in \L} \D_\rho\ww_{\xi,\rho}(x)
    V_\rho = 0.
  \end{equation}
  We will now independently estimate the two remaining groups
  involving $T_1$ and $E_1$.

  {\it Estimating $T_1$: } We expand
  \begin{align*}
    D_\vsig u(\xi)-\D_\vsig u(x)
    =\,& \b(\D_{\xi-x}\D_\vsig + \smfrac12 \D_{\vsig}^2 + \smfrac12 \D_{\xi-x}^2 \D_\vsig + \smfrac12 \D_{\xi-x}
    \D_\vsig^2 + \smfrac16 \D_\vsig^3\b) u(x) + E_3(\xi,\vsig), \\
    \text{where} \quad |E_3(\xi,\vsig)| \lesssim\,& (|\rho|+|\vsig|)^3
    |\vsig| \mu_{\xi,\vsig}.
  \end{align*}
  Summing over $\xi,\rho$, applying
  \eqref{eq:dS_symm_0}--\eqref{eq:dS_symm_2}, and applying the
  symmetry $V_{\rho\vsig} = V_{\vsig\rho}$, yields
  \begin{align}
    \notag
    \sum_{\rho\in\Rg} \sum_{\xi\in\L} \D_\rho \ww_{\xi,\rho}
    T_1(\xi,\rho) =\,& \sum_{\rho,\vsig\in\Rg} V_{\rho\vsig}\b[\;\cdot\;, (\D_\rho\D_\vsig - \smfrac12 \D_\rho^2
    \D_\vsig + \smfrac12 \D_\rho\D_\vsig^2) u(x)] + E_3' \\
    \label{eq:dS:20}
    =\,& \sum_{\rho,\vsig\in\Rg} V_{\rho\vsig}
    [\;\cdot\;, \D_\rho\D_\vsig u] + E_3' = {\rm div} \Scb + E_3', \\
    \notag
    \text{where} \quad |E_3'| \lesssim\,& \sum_{\bfrho\in\Rg^2}
    \md^{(\infty)}(\bfrho) \, \| \D^4 u \|_{L^\infty(x+\nu_{\rho_1,\rho_2})},
  \end{align}
  and where the estimate of $|E_3'|$ is analogous to the estimate of
  $|E_2'|$ in \eqref{eq:dS:E2'}.

  {\it Estimating $E_1$: } To estimate $E_1$ we use the simpler
  expansion \eqref{eq:stress:c2_10}, which can be written as
  \begin{equation}
    \label{eq:dS_24}
    \begin{split}
      D_\vsig u(\xi) - \D_\vsig u =\,& \D_{\xi-x}\D_\vsig u + \smfrac12
      \D_\vsig^2 u + \delta_{\xi,\vsig}', \\
      \text{where} \quad |\delta_{\xi,\vsig}'| \lesssim\,&
      (|\rho|+\vsig|)^2 |\vsig| \delta_{\xi,\vsig}.
    \end{split}
  \end{equation}
  We apply again \eqref{eq:dS_symm_0}--\eqref{eq:dS_symm_2}, followed
  by the symmetry $V_{\rho\vsig\tau} = - V_{-\rho,-\vsig,-\tau}$, to
  obtain
  \begin{align}
    \notag
    \sum_{\rho\in\Rg} \sum_{\xi\in\L} \D_\rho \chi_{\xi,\rho}
    E_1(\xi,\rho) =\,& \sum_{\rho,\vsig,\tau\in\Rg} \sum_{\xi\in\L}
    \D_\rho \ww_{\xi,\rho} V_{\rho\vsig\tau}\b[\;\cdot\;,
    \D_{\xi-x}\D_\vsig u + \smfrac12 \D_\vsig^2 u, \D_{\xi-x}\D_\tau u
    + \smfrac12 \D_\tau^2 u\b] + E_1' \\
    \notag
    =\,& \sum_{\rho,\vsig,\tau\in\Rg} \B( -\smfrac12
    V_{\rho\vsig\tau}[\;\cdot\;, \D_\rho\D_\vsig u,
    \D_\rho\D_\tau u]
    + \smfrac12 V_{\rho\vsig\tau}[\;\cdot\;, \D_\rho\D_\vsig u,
    \D_\tau^2 u]  \\[-2mm]
    \notag
    & \hspace{2.5cm} + \smfrac12 V_{\rho\vsig\tau}[\;\cdot\;,\D_\vsig^2u,
    \D_\rho\D_\tau u] + \smfrac14
    V_{\rho\vsig\tau}[\;\cdot\;,\D_\vsig^2 u, \D_\tau^2 u] \B) + E_1' \\
    \label{eq:dS:25}
    =\,& E_1',
  \end{align}
  where
  \begin{align*}
    E_1' =\,& \sum_{\rho,\vsig,\tau\in\Rg} \sum_{\xi\in\L}
    \D_\rho \ww_{\xi,\rho} V_{\rho\vsig\tau}\b[\;\cdot\;,
    \D_{\xi-x}\D_\vsig u + \smfrac12 \D_\vsig^2 u, \delta_{\xi,\tau}'
    u\b] \\
    & + \sum_{\rho,\vsig,\tau\in\Rg} \sum_{\xi\in\L}
    \D_\rho \ww_{\xi,\rho} V_{\rho\vsig\tau}\b[\;\cdot\;,
    \delta_{\xi,\vsig}', \D_{\xi-x}\D_\tau u
    + \smfrac12 \D_\tau^2 u\b] \\
    & + \sum_{\rho,\vsig,\tau\in\Rg} \sum_{\xi\in\L}
    \D_\rho \ww_{\xi,\rho}
    V_{\rho\vsig\tau}\b[\;\cdot\;,\delta_{\xi,\vsig}',
    \delta_{\xi,\tau}'\b] \\
    =:\,& E_{1,1}' + E_{1,2}' + E_{1,3}'.
  \end{align*}
  The first and second terms can be treated in a straightforward
  manner, employing \eqref{eq:dS_24}:
  \begin{align*}
    |E_{1,1}'| \lesssim\,& \sum_{\bfrho\in\Rg^3}
    \md^{(\infty)}(\bfrho) |\D^2 u(x)| \|\D^3 u
    \|_{L^\infty(x+\nu_{\rho_1,\rho_3})}, \quad \text{and} \\
    |E_{1,2}'| \lesssim\,& \sum_{\bfrho\in\Rg^3}
    \md^{(\infty)}(\bfrho)|\D^2 u(x)| \|\D^3 u
    \|_{L^\infty(x+\nu_{\rho_1,\rho_2})}.
  \end{align*}
  However, the term $E_{1,3}'$ must be treated more carefully. Note
  that simply using \eqref{eq:dS_24} to estimate
  \begin{displaymath}
        |E_{1,3}'| \lesssim \sum_{\bfrho\in\Rg^3}
    \frac{m(\bfrho)|\bfrho|^4}{|\rho_1|}  \|\D^3 u
    \|_{L^\infty(x+\nu_{\rho_1,\rho_2})} \|\D^3 u
    \|_{L^\infty(x+\nu_{\rho_1,\rho_3})}
  \end{displaymath}
  would impose a decay on the interaction potential that would rule
  out the Lennard-Jones potential. Instead, we use \eqref{eq:dS_24}
  only for the terms $\delta_{\xi,\vsig}'$ and the weaker estimate
  \begin{align*}
    |\delta_{\xi,\tau}'| =\,& \b| D_{\tau} u(\xi) - \D_\tau u -
    \D_{\xi-x} \D_\tau u - \smfrac12 \D_\tau^2 u \b|
    \lesssim (|\rho|+|\tau|) |\tau| \eps_{\xi, \tau},
  \end{align*}
  which leads to
  \begin{displaymath}
    |E_{1,3}'| \lesssim \sum_{\bfrho\in\Rg^3}
    \frac{m(\bfrho)|\bfrho|^3}{|\rho_1|}  \|\D^3 u
    \|_{L^\infty(x+\nu_{\rho_1,\rho_2})} \|\D^2 u
    \|_{L^\infty(x+\nu_{\rho_1,\rho_3})}.
  \end{displaymath}
  Note that the coefficients in this last estimate are precisely
  $\md^{(\infty)}(\bfrho)$, which are controlled also for the
  Lennard-Jones case. Combining the estimates for $E_{1,1}', E_{1,2}'$
  and $E_{1,3}'$, we arrive at
  \begin{equation}
    \label{eq:dS_35}
    |E_1'| \lesssim \sum_{\bfrho\in\Rg^3}
    \md^{(\infty)}(\bfrho) \| \D^3 u
    \|_{L^\infty(x+\nu_{\rho_1,\rho_2})} \| \D^2 u \|_{L^\infty(x+\nu_{\rho_1,\rho_3})}.
  \end{equation}
  Note that we have seemingly ignored the term $E_{1,1}'$, however,
  due to the symmetry $m(\rho_1,\rho_2,\rho_3) =
  m(\rho_1,\rho_3,\rho_2)$, this term is in fact included.

  {\it Combining the estimates: } Combining \eqref{eq:dS:E2'},
  \eqref{eq:dS:20} and \eqref{eq:dS_35} we obtain the desired upper
  bound on $|{\rm div} \Sa - {\rm div} \Scb|$.
\end{proof}

Repeating the arguments of the proof of Lemma
\ref{th:stress:moderr_negSob} almost verbatim, we obtain the following
global consistency error estimate.

\begin{lemma}
  \label{th:dS_global}
  Let $u \in \WWh^{4,p}$, $\tilu := \zz \ast u|_\L$ and $v \in
  \ell^{p'}$; then
  \begin{align*}
    \b|\b\< \del\Ea(\tilu), \tilv \b\> - \b\< \del\Ec(u), \barv \b\>
    \b| \lesssim \b(\,& \Md^{(2,p)} \| \D^4 u \|_{L^p} + \Md^{(3,p)} \| \D^3 u \|_{L^{3p/2}} \| \D^2 u \|_{L^{3p}} \\
    & + \Md^{(4,p)} \| \D^2 u \|_{L^{3p}}^3 \b) \| \barv
    \|_{L^{p'}}.
  \end{align*}
\end{lemma}

\medskip
We are now in a position to prove \eqref{eq:w:prf_frc_cons}.

\begin{proof}[Proof of \eqref{eq:w:prf_frc_cons}]
  Recall that
  \begin{displaymath}
    (\beta, \D\dot{w}) = \< \del\Ec(u^\c), \dot{w} \> - \< \del\Ea(z),
    \dot{e} \>,
  \end{displaymath}
  where $\tilw = e$ and $z = (\zz \ast u^\c)|_\L$. Hence, Lemma
  \ref{th:dS_global} and \eqref{eq:w:cb_reg} yield
  \begin{displaymath}
    \b|(\beta, \D\dot{w})\b| \leq C \eps^{3-d/2} \| \dot w \|_{L^2},
  \end{displaymath}
  where $C$ depends only on $\Md^{(j,2)}$, $j = 2, 3, 4$, and is
  otherwise generic. This proves \eqref{eq:w:prf_frc_cons_pre}, from
  which \eqref{eq:w:prf_frc_cons} follows immediately after
  application of Cauchy's inequality.
\end{proof}

\subsection{Estimating the nonlinearity}
\label{sec:w:nonlin}
In this section, we prove two estimates on the term $\< \dot{H} e, e
\>$ occurring in the error equation \eqref{eq:w:err_eqn_2}.

\begin{proof}[Proof of \eqref{eq:w:prf_nonlin}]
  Recall \eqref{eq:w:Ea_smooth_t1}, which states that $z(t) + \theta
  e(t) \in \Ba$ for all $t \in [0, s]$ and $\theta \in [0, 1]$.

  We write out $\< \dot{H} e, e\>$ explicitly and exchange the order
  of integration (this is justified since $M^{(3)}$ is finite, which
  implies that the integrand is uniformly bounded),
  \begin{align*}
    \int_0^s \< \dot{H} e, e \> \dt =\,& \int_0^s \int_0^1
    \del^3\Ea(z+\theta e)[\dot{z} + \theta \dot{e}, e, e] \dd\theta\, \dt \\
    =\,& \int_0^1 \int_0^s \del^3 \Ea(z+\theta e)[\dot{z}, e, e]\, \dt
    \dd\theta
    + \int_0^1 \theta \int_0^s \del^3\Ea(z+\theta e)[\dot{e}, e, e]
    \dt\, \dd\theta.
  \end{align*}
  The first group on the right-hand side is easily estimated, using
  \eqref{eq:bound_delj} and \eqref{eq:w:cb_reg},
  \begin{displaymath}
    \del^3 \Ea(z+\theta e)[\dot{z}, e, e] \lesssim M^{(3)} \| \D \dot{z}
    \|_{L^\infty} \| \D e \|_{L^2}^2 \lesssim \eps M^{(3)} \| \D\dot U^\c \|_{L^\infty} E^2.
  \end{displaymath}

  To estimate the second group, we focus on the inner integral
  only. Since $\del^3\Ea(\;\cdot\;)[v_1, v_2, v_3]$ is invariant under
  permutation of the arguments $v_1, v_2, v_3$ it follows that
  \begin{displaymath}
    \smfrac{\dd}{\dt} \smfrac13 \del^3\Ea(z+\theta e)[e, e, e]
    = \del^3\Ea(z+\theta e)[\dot{e}, e, e] + \smfrac{1}{3} \del^4\Ea(z+\theta
    e)[\dot{z}+\theta\dot{e}, e, e, e],
  \end{displaymath}
  which gives
  \begin{equation}
    \label{eq:w:nonlin:int1}
    \int_0^s \del^3\Ea(z+\theta e)[\dot{e}, e, e]
    \dt = \smfrac13 \b[ \del^3\Ea(z+\theta e)[e, e, e] \b]_{t =
      0}^s - \smfrac{1}{3} \int_0^s \del^4\Ea(z+\theta
    e)[\dot{z}+\theta\dot{e}, e, e, e] \dt.
  \end{equation}
  Applying \eqref{eq:bound_delj} and the inverse estimate Lemma
  \ref{th:embedding} with $j = 0, m = 1$ and $p=\infty$, $q=2$ to
  bound $\| \D \dot{e} \|_{L^\infty} \lesssim \| \dot{e} \|_{L^2}$, we
  obtain
  \begin{align*}
     \int_0^s \del^3\Ea(z+\theta e)[\dot{e}, e, e] \dt
     \lesssim\,& M^{(3)} \b( \| \D\bare(0)\|_{L^2}^3 +
     \| \D \bare(s) \|_{L^2}^3 \b) + M^{(4)} \int_0^s \| \D\dot{z}
     \|_{L^\infty} \| \D\bare \|_{L^2}^3 \\
     & +  M^{(4)} \int_0^s
     \|\D\dot{e}\|_{L^\infty} \| \D\bare \|_{L^2}^3 \dt \\
     \lesssim\,& M^{(3)} \b( E(0)^3 + E(s)^3 \b) +
     M^{(4)} \int_0^s \B(\eps E^3 + E^4 \B) \dt. \qedhere
  \end{align*}
\end{proof}

\begin{proof}[Proof of \eqref{eq:w:prf_nonlin_5}]
  To prove \eqref{eq:w:prf_nonlin_5} we may now use the fact that $k
  \geq 5$, that is, $\Ea$ is five times differentiable in a
  neighborhood of $z(t) + \theta e(t)$, $\theta \in [0, 1], t \in [0,
  t_1]$. Thus, starting from \eqref{eq:w:nonlin:int1} we can repeat
  the integration by parts argument to obtain
  \eqref{eq:w:prf_nonlin_5}.
\end{proof}

\section{Conclusion}
In the small deformation or short time regimes we have developed an
essentially complete approximation theory of the Cauchy--Born model
for Bravais lattices for both static and dynamic problems. The main
open questions are 1. the extension to large deformations and/or the
charactisation of maximal time intervals for which the error estimates
hold; and 2. the extension to multi-lattices.

\section*{Acknowledgements}
We thank E. S\"{u}li for discussions to help clarify the properties of
homogeneous Sobolev spaces; C. Makridakis and E. S\"{u}li for
discussions on the error analysis for nonlinear wave equations during
two visits of C. Ortner to the University of Crete funded by the
Archimedes Center for Modeling, Analysis and Computation; and
A. Shapeev for suggesting a variant of the ``localisation trick''
\eqref{eq:intro_localisation_trick} during a research visit of
C. Ortner to the EPFL funded by the Chair of A. Abdulle.

\appendix

\section{Auxiliary Approximation and Interpolation Results}
%
In this appendix we collect several auxiliary results related to the
smoothing properties of the convolution operator $f \mapsto \zz \ast
f$, and the accuracy of the approximation $\zz\ast f$ to $f$.

\begin{lemma}
  \label{th:convolute_affine}
  Let $\zz$ be the basis function defined in \S\ref{sec:interp}, and
  $f(x) = a + b \cdot x$ where $a \in \R, b \in \R^d$ are constants,
  then $\zz \ast f = f$.
\end{lemma}
\begin{proof}
  Since $\int \zz(x) \dx = 1$, it follows that $\zz \ast a =
  a$. Further,
  \begin{align*}
    \int_\Om \zz(x-y) y \dy =\,& x - \int_\Om \zz(x-y) (x-y) \dy
    = x - \int_\Om \zz(-y) \cdot (-y) \dy = x,
  \end{align*}
  where we used that fact that $\zz$ is an even function.
\end{proof}

\begin{lemma}
  \label{th:convolution_error}
  Let $f \in \WW^{2,p}$, $p \in [1, \infty]$, then
  \begin{displaymath}
    \| \zz\ast f - f \|_{L^p} \lesssim \| \D^j f \|_{L^p}, \qquad
    \text{for } j = 0, 1, 2.
  \end{displaymath}
\end{lemma}
\begin{proof}
  This result follows immediately from Lemma \ref{th:convolute_affine}
  and Poincar\'{e} inequalities.
\end{proof}

\begin{corollary}
  \label{th:interp_error}
  Let $u \in \WWh^{3,p}$ with $p > d / 3$ and let $\tilu := \zz \ast
  u$, then
  \begin{displaymath}
    \| \D u - \D \Itil \tilu \|_{L^p} \lesssim \| \D^3 u \|_{L^p}.
  \end{displaymath}
\end{corollary}
\begin{proof}
  Under the assumption that $p > d/3$ (which guarantees the embedding
  $\WWh^{3,p} \subset C$) we know from \cite{OrShSu:2012} that
  \begin{displaymath}
    \| \D \tilu - \D \Itil \tilu \|_{L^p} \lesssim \| \D^3 \tilu
    \|_{L^p} \lesssim \| \D^3 u \|_{L^p}.
  \end{displaymath}

  Applying also Lemma \ref{th:convolution_error} with $f = \D u$, we
  obtain
  \begin{displaymath}
    \| \D u - \D \Itil \tilu \|_{L^p} \leq \| \D u - \D \tilu \|_{L^p}
    + \| \D \tilu - \D \Itil \tilu \|_{L^p}
    \lesssim \| \D^3 u \|_{L^p}.  \qedhere
  \end{displaymath}
\end{proof}

\begin{lemma}
  \label{th:conv_est}
  Let $\nu \subset \R^d$ be measurable, $-\nu = \nu$, and $f \in
  L^p(\Om)^j$, then
  \begin{displaymath}
    \int_\Om \| \zz \ast f \|_{L^\infty(x+\nu)}^p \dx \leq {\rm vol}(\nu') \, \| f \|_{L^p}^p,
  \end{displaymath}
  where $\nu' := \bigcup_{x \in \nu} \supp(\zz(x-\;\cdot\;))$.
\end{lemma}
\begin{proof}
  Let $\zz'(x, z) := \max_{y \in x + \nu} \zz(y-z)$.  Since $\zz \geq
  0$ and $\int_\Om \zz \dx = 1$, Jensen's inequality yields
  \begin{align*}
    \| \zz \ast f \|_{L^\infty(\nu(x))}^p =\,& \max_{y \in x + \nu}
    \bg|\int_\Om \zz(y - z) f(z) \dz \bg|^p \leq \max_{y \in x + \nu}
    \int_\Om \zz(y-z) |f(z)|^p \dz\\
    \leq\,& \int_\Om \max_{y \in x + \nu} \zz(y-z) |f(z)|^p \dz
    = \int_\Om \zz'(x,z) |f(z)|^p \dz.
  \end{align*}
  Integrating with respect to $x$, we obtain
  \begin{align*}
    \int_\Om \| \zz \ast f \|_{L^\infty(x+\nu)}^p \dx
    \leq\,& \int_\Om \int_\Om \zz'(x,z) |f(z)|^p \dz \dx \\
    =\,& \int_\Om |f(z)|^p \int_\Om  \zz'(x,z) \dx \dz.
  \end{align*}

  From its definition it is clear that $0 \leq \zz' \leq 1$. Moreover,
  if $\zz'(x, z) \neq 0$, then $\zz(y - z) \neq 0$ for some $y \in
  x+ \nu$, that is,
  \begin{displaymath}
    y - x \in \nu \quad \text{and} \quad z - y \in \supp\zz.
  \end{displaymath}
  Since both $\nu$ and $\supp\zz$ are symmetric about the origin, this
  implies that $x - z \in \bigcup_{y \in \nu} \supp(y - \;\cdot\;) =:
  \nu'$. Thus, we obtain $\int_\Om \zz'(x,z) \dx \leq \vol(\nu')$, which
  concludes the proof.
\end{proof}

\begin{lemma}
  \label{th:ex_frc}
  Let $f^\c \in L^1_\loc \cap \WWh^{-1,2}$ and let $f^\a \in \Us$ be defined by
  \begin{equation}
    \label{eq:errsm:defn_fa}
    f^\a(\xi) := \int_{\Om} \zz(\xi - x) f^\c(x) \dx,
  \end{equation}
  then $f^\a \in \Ush^{-1,2}$. Moreover, if $\D f^\c \in L^2$, then
  \begin{equation}
    \label{eq:fa:err}
    \b| (f^\a, \tilv)_{\L} - (f^\c, \barv)_{\R^d} \b| \leq \| \D
    f^\c \|_{L^2} \| \D\barv \|_{L^2} \qquad \forall v \in \Usz.
  \end{equation}
\end{lemma}
\begin{proof}
  Let $v \in \Usz$, then
  \begin{align*}
    (f^\a, v)_{\L}
    =\, \sum_{\xi \in \L} f^\a(\xi) \cdot v(\xi)
    =\, \int_{\R^d} f^\c(x)  \sum_{\xi \in \L} \zz(x-\xi) v(\xi) \dx
    =\, (f^\c, \barv)_{\Om},
  \end{align*}
  which implies that $f^\a \in \Ush^{-1,2}$ with
  $\|f^\a\|_{\Ush^{-1,2}} \leq \| f^\c \|_{\WWh^{-1,2}}$.

  To prove \eqref{eq:fa:err}, we first note that
  \begin{align}
    \notag
    (f^\a, \tilv)_{\L}
    =\,& \sum_{\xi \in \L} f^\a(\xi) \int_\Om \zz(x-\xi)
    \barv(x) \dx \\
    \label{eq:fa:1}
    =\,& \int_\Om \barv(x) \cdot \sum_{\xi \in \L} \zz(x-\xi)
    f^\a(\xi) \dx = (\barf^\a, \barv)_{\Om}.
  \end{align}
  This allows us to write
  \begin{displaymath}
    (f^\a, \tilv )_{\L} - (f^\c, \barv)_{\Om} = \sum_{\xi \in \L}
    \int_{\R^d} \zz(\xi-x) (f^\a(\xi) - f^\c(x)) {v}(x) \dx.
  \end{displaymath}
  From the definition of $f^\a$  \eqref{eq:errsm:defn_fa} it follows
  that
  \begin{displaymath}
    \int \zz(\xi-x) (f^\a(\xi) - f^\c(x)) = 0,
  \end{displaymath}
  and hence we obtain, for some arbitrary constants $c_\xi \in \R$,
  \begin{align*}
    \b| (f^\a, \tilv )_{\L} - (f^\c, \barv)_{\Om} \b|
    =\,& \bg| \sum_{\xi \in \L} \int_{\R^d} \zz(\xi-x) \b(f^\a(\xi) - f^\c(x)\b)
    \b({u}(x) - c_\xi \b) \ds \bg| \\
    \leq\,& \sum_{\xi \in \L} \bg( \int_{\R^d} \zz(\xi-x)^2 \b| f^\a(\xi) -
    f^\c(x) \b|^2 \dx \bg)^{1/2} \, \b\| \barv  - c_\xi \b\|_{\LL^2(Q_\xi)}.
  \end{align*}

  Choosing $c_\xi = (\barv)_{Q_\xi}$ and applying Poincar\'{e}'s
  Inequality, we can estimate
  \begin{displaymath}
    \b\| \barv  - c_\xi \b\|_{\LL^2(Q_\xi)} \leq
    \b(\smfrac{2}{\pi}\b)^d \b\| \D\barv \b\|_{\LL^2(Q_\xi)}.
  \end{displaymath}
  Moreover, estimating $\zz \leq 1$, and using the fact that
  $f^\a(\xi)$ is the orthogonal projection of $f^\c$ with respect to the
  kernel $\zz(\xi - \;\cdot\;)$, we obtain
  \begin{align*}
    \bg( \int \zz(\xi-x)^2 \b| f^\a(\xi) -
    f^\c(x) \b|^2 \dx \bg)^{1/2}
    \leq\,& \bg( \int \zz(\xi-x) \b| f^\a(\xi) -
    f^\c(x) \b|^2 \dx \bg)^{1/2}  \\
    \leq\,& \bg( \int \zz(\xi-x) \b| (f^\c)_{Q_\xi} -
    f^\c(x) \b|^2 \dx \bg)^{1/2}  \\
    \leq\,& \b\| (f^\c)_{Q_\xi} - f^\c \b\|_{\LL^2(Q_\xi)}
    \leq \b(\smfrac{2}{\pi}\b)^d \b\| \D
    f^\c \b\|_{\LL^2(Q_\xi)}.
  \end{align*}
  Combining the foregoing estimates and estimating the overlaps we
  arrive at
  \begin{displaymath}
    \b| (f^\a, \tilv )_{\L} - (f^\c, \barv)_{\Om} \b| \leq C
    \| \D f^\c \|_{\LL^2} \| \D \barv \|_{\LL^2}.
  \end{displaymath}
  with $C = 2^{3d} / \pi^{2d} \leq 1$ for $d \in \{1,2,3\}$. This
  establishes \eqref{eq:fa:err}.
\end{proof}

\section{Examples of admissible potentials}
\label{sec:examples}
We discuss the most common interatomic potentials and show that they
can be accommodated within our framework. We remark from the outset
that our smoothness requirement (at least four times continuously
differentiable for the error analysis) is reasonable for physical
interaction potentials, but is not satisfied by typical potentials
constructed for molecular dynamics simulations, which employ cut-off
functions that are often only once differentiable.

\subsection{Lennard-Jones type potentials}
\label{sec:ex_lj}
For a pure pair interaction model, we define
\begin{equation}
  \label{eq:ex_lj_V}
  V(\bfg) := \frac12 \sum_{\rho \in \Rg} \b[\varphi\b(|g_\rho|\b) -
  \varphi\b( |\mA\rho|\b)\b].
\end{equation}
With this definition, $V$ clearly satisfies the symmetry
\eqref{eq:intro:ptsymm}.

The prototypical example is of course the Lennard--Jones potential
\cite{LennardJones:1924a},
\begin{displaymath}
  \varphi(r) = r^{-12} - 2 r^{-6}.
\end{displaymath}
In this case, one readily sees that $\varphi^{(j)}(r) \lesssim
r^{-6-j}$, for $r \geq 1$ and $j \in \N$. More generally, suppose that
$V$ is of the form \eqref{eq:ex_lj_V} with
\begin{equation}
  \label{eq:ex_lj_decay}
  |\varphi^{(j)}(r)| \lesssim r^{-\alpha - j}, \quad \text{ for } r
  \geq 1, \quad j \in \N,
\end{equation}
then one may readily deduce that
\begin{displaymath}
  m(\bfrho) \lesssim \cases{ |\rho|^{-\alpha}, & \bfrho = (\rho,
    \dots, \rho) \in \Rg^j, \\
    0, & \text{otherwise},}
\end{displaymath}
and consequently,
\begin{displaymath}
  M^{(j)} \lesssim \sum_{\rho\in\Rg} |\rho|^{-\alpha} \qquad
  \text{and} \qquad
  \Ms^{(j,2)} + \Md^{(j,2)} \lesssim \sum_{\rho\in\Rg} |\rho|^{-\alpha+5/2}.
\end{displaymath}
Thus, $M^{(j)}$ is finite (and hence $\Ea$ well-defined and $k$ times
differentiable) if and only if $\alpha > d$; and $\Md^{(j,2)},
\Ms^{(j,2)}$ are finite (and hence our error analysis applies) if and
only if $\alpha > d+5/2$.

In particular, it follows that the Lennard-Jones potential is included
in our analysis. Another commonly employed potential is the Morse
potential \cite{Morse:1929a}, which decays exponentially and is hence
trivially included our analysis. The Coulomb potential, $\varphi(r) =
r^{-1}$, is excluded.

\subsection{Embedded atom method}
\label{sec:ex_eam}
In the embedded atom method \cite{Daw:1984a} one postulates site
energies of the form
\begin{equation}
  \label{eq:ex_eam}
  V(\bfg) = \sum_{\rho\in\Rg} \varphi(|g_\rho|) + G\b( {\textstyle
    \sum_{\rho\in\Rg} \psi(|g_\rho|)} \b),
\end{equation}
where $\varphi$ is a Lennard-Jones or Morse type pair potential,
$\psi(|g_\rho|)$ is a model of the electron density at $0$ generated
by a nucleus at distance $|g_\rho|$, and $G$ is the energy to embed a
nucleus into a see of electrons. Again, it is clear from the
functional form of $V$, that it satisfies the symmetry
\eqref{eq:intro:ptsymm}.

The computation of the partial derivatives is now more
involved. Suppose, for simplicity, that $\varphi \equiv 0$, define
$\bar\psi := \sum_{\rho\in\Rg} \psi(|g_\rho|)$, and $\Psi(g) :=
\psi(|g|)$, then
\begin{align*}
  V_\rho(\bfg) =\,& G'(\bar\psi) \D\Psi(g_\rho), \\
  V_{\rho\vsig}(\bfg) =\,& G''(\bar\psi) \D\Psi(g_\rho) \otimes
  \D\Psi(g_\vsig) + G'(\bar\psi) \D^2 \Psi(g_\rho)
  \delta_{\rho,\vsig}, \\
  V_{\rho\vsig\tau}(\bfg) =\,&  G'''(\bar\psi) \D\Psi(g_\rho) \otimes
  \D\Psi(g_\vsig) \otimes \D\Psi(g_\tau)  + G'(\bar\psi) \D^3
  \Psi(g_\rho) \delta_{\rho,\vsig} \delta_{\rho,\tau} \\
  & \hspace{-2cm} + G''(\bar\psi)\B(
  \D^2\Psi(g_\rho)\otimes \D\Psi(g_\vsig) \delta_{\rho,\tau} +
  \D\Psi(g_\rho) \otimes \D^2\Psi(g_\vsig) \delta_{\vsig,\tau} +
  \D^2\Psi(g_\rho)\otimes \D\Psi(g_\tau) \delta_{\rho,\vsig} \B),
\end{align*}
and so forth. $G$ is typically chosen smooth and $\varphi, \psi$ decay
exponentially. In that case, one immediately sees that all constants
$M^{(j)}, \Ms^{(j,p)}, \Md^{(j,p)}$ are bounded. More generally, let
$\varphi \equiv 0$ and suppose that
\begin{equation}
  \label{eq:decay_psi}
  \psi^{(j)}(r) \lesssim r^{-\beta-j}.
\end{equation}
In this case, also $|\D^j \Psi(g)| \lesssim |g|^{-\beta-j}$, and we
conclude that
\begin{align*}
  m(\rho) \lesssim\,& |\rho|^{-\beta}, \qquad \rho\in\Rg, \\
  m(\rho,\vsig) \lesssim\,& |\rho|^{-\beta} |\vsig|^{-\beta} +
  |\rho|^{-\beta} \delta_{\rho\vsig}, \qquad \rho,\vsig\in\Rg, \\
  m(\rho,\vsig,\tau) \lesssim\,&
  |\rho|^{-\beta}|\vsig|^{-\beta}|\tau|^{-\beta} + |\rho|^{-\beta}
  \delta_{\rho,\vsig}\delta_{\rho,\tau} \\
  & + \B( |\rho|^{-\beta} |\vsig|^{-\beta} \delta_{\rho,\vsig} +
  |\rho|^{-\beta}|\tau|^{-\beta} \delta_{\rho,\tau} +
  |\vsig|^{-\beta} |\tau|^{-\beta} \delta_{\vsig,\tau} \B), \quad
  \rho,\vsig,\tau\in\Rg,
\end{align*}
and so forth. Due to the product structure, one can readily see that
$M^{(j)}$ is finite provided that $\beta > d$.

However, to ensure that $\Md^{(j,2)}, \Ms^{(j,2)}$ are finite, we now
require more stringent requirements. For example, considering only the
first group in $m(\rho,\vsig,\tau)$ and indicating the missing terms
by ``$\dots$'', and using $|\rho\times\vsig| \leq |\rho||\vsig|$, we
can estimate
\begin{align*}
  \Ms^{(3,2)} \lesssim\,& \sum_{\rho,\vsig,\tau\in\Rg}
  (|\rho|+|\vsig|+|\tau|)^3 |\rho|^{-\beta} |\vsig|^{-\beta}
  |\tau|^{-\beta} + \dots \\
  \lesssim\,& \sum_{\rho,\vsig,\tau \in\Rg}
  \b[|\rho|^{3-\beta}|\vsig|^{-\beta}|\tau|^{-\beta} +
  |\rho|^{2-\beta}|\vsig|^{1-\beta}|\tau|^{-\beta} + \dots \b] +
  \dots,
\end{align*}
which is finite provided that $\beta > d + 3$. The remaining terms
can be treated analogously. For the dynamic case, the extra factor
$|\rho_1|^{-1}$ does not help except in the case of pair interactions,
and we require $\beta > d + 4$ to ensure that the constants
$\Md^{(j,2)}$ are finite.

In summary, if $V$ is of the form \eqref{eq:ex_eam} with the pair
interaction $\varphi$ satisfying \eqref{eq:ex_lj_decay} and the
electron density function $\psi$ satisfying \eqref{eq:decay_psi}, then
we require $\alpha, \beta > d$ to ensure that the constants $M^{(j)}$,
$1 \leq j \leq k$ are finite; we require $\alpha > d+5/2, \beta > d+3$
to ensure that $\Ms^{(j,2)}$, $j = 2, 3$, are finite; and we require
$\alpha > d + 5/2, \beta > d + 4$ to ensure that $\Md^{(j,2)}$, $2
\leq j \leq 4$ are finite.

\subsection{Bond-angle potentials}
\label{sec:ex_short}
Lennard-Jones type pair interactions and embedded atom potentials are
the prototypical long-ranged potentials with infinite interaction
range. Most other potentials used in molecular simulations act only on
a finite interaction neighborhood. For example, bond-angle potentials
(3-body or 4-body) act only on angles between nearest neighbors. We
only need to check whether they can be written in a way that preserves
the inversion symmetry~\eqref{eq:intro:ptsymm}. 3-body bond-angle
energies are typically written in the form
\begin{displaymath}
  \sum_{\xi,\eta,\mu\in\L}
  \varphi(|r_{\eta\xi}|)\varphi(|r_{\mu,\xi}|) \psi(\theta_{\eta\xi\mu}),
\end{displaymath}
where $r_{\eta\xi} = \eta+u(\eta)-\xi-u(\xi)$, $\theta_{\eta\xi\mu}$
is the angle between the bond directions $r_{\eta\xi}, r_{\mu\xi}$,
$\varphi$ is a cut-off function to ensure that the potential acts only
on nearest-neighbors, and $\psi$ is an angle potential that drives
towards preferred bond-angles. This term is symmetric about the
center-atom, which suggests to write
\begin{displaymath}
  V(\bfg) = \sum_{\substack{\rho,\vsig\in\Rg \\ \rho \neq \vsig}}
  \varphi(|r_\rho|) \varphi(|r_\vsig|) \psi(\theta_{\rho\vsig}),
\end{displaymath}
where $r_\rho := \rho+g_\rho$, and $\theta_{\rho\vsig}$ is the angle
between $r_\rho, r_\vsig$. This sum is fully permutation invariant,
and hence the inversion symmetry \eqref{eq:intro:ptsymm} holds.

4-body (or, dihedral angle; or, torsion) potentials can be treated
similarly. There are now two center atoms in bonds of this type, and
hence one ``splits'' the bond between the two corresponding site
energies (similarly as in the pair potential case). By summing over
all quadruples involved with the given site, the resulting site
potential will again be permutation invariant.

\subsection{Generic multi-body potentials}
A more recent development are potentials without physical
interpretation, but simply postulating a general functional form for
$V$, and fitting a large number of parameters to energy and forces
obtained from electronic structure calculations; see, e.g.,
\cite{BaPaKoCs:GAP}. Such general potentials are normally constructed
to satisfy the permutation invariance, and hence the inversion symmetry
\eqref{eq:intro:ptsymm}, and are therefore still included in our
analysis.

\section{Lattice Stability versus Ellipticity}
\label{sec:app_stab}
We show that the lattice stability assumption \eqref{eq:defn_Lam} is
not only sufficient but also {\em necessary} to obtain Theorems
\ref{th:errsm:mainthm} and \ref{th:w:mainthm}. This can already be
seen for 1D second-neighbour harmonic pair interactions:
\begin{displaymath}
  \Es_\xi(u) = \smfrac{a_1}{4} \b( |u_\xi'|^2 + |u_{\xi+1}'|^2 \b)
  + \smfrac{a_2}{4} \b( |u_{\xi-1}'+u_\xi'|^2 + |u_{\xi+1}'+u_{\xi+2}'|^2 \b),
\end{displaymath}
where $u_\eta' = u_\eta - u_{\eta-1}$. In this case, the atomistic and
Cauchy--Born energies are more conveniently written in the form
\begin{align*}
  \Ea(u) = \sum_{\xi\in\L} \B(  \smfrac{a_1}{2} |u_\xi'|^2 +
  \smfrac{a_2}{2} |u_\xi'+u_{\xi+1}'|^2 \B) \quad \text{and} \quad
  \Ec(u) = \smfrac{(a_1+4a_2)}{2}\int_\R |u'|^2\,\dd x.
\end{align*}
We consider two choices for the coefficients $a_1, a_2$:
\begin{displaymath}
  \begin{array}{rcl}
    a_1^{(1)} &=& 2, \\
    a_2^{(1)} &=& -\smfrac{1}{4},
  \end{array}
  \qquad \text{and} \qquad
  \begin{array}{rcl}
    a_1^{(2)} &=& -1, \\
    a_2^{(2)} &=& \smfrac{1}{2},
  \end{array}
\end{displaymath}
then in both of these cases we have
\begin{displaymath}
  a_1^{(j)} + 4 a_2^{(j)} = 1.
\end{displaymath}
Thus, the continuum energy is positive definite and hence the
continuum wave equation is well-posed.

In the atomistic case, we can use the parallelogram formula to rewrite
\begin{equation}
  \label{eq:app_stab_Ea}
  \Ea(u) = \sum_{\xi\in\L} \B( \smfrac{a_1+4a_2}{2} |u_\xi'|^2 -
  \smfrac{a_2}{2} |u_\xi''|^2 \B),
\end{equation}
where $u_\xi'' = u_{\xi+1} - 2 u_{\xi} + u_{\xi-1}$.  Hence, in the
case $a_i = a_i^{(1)}$ we have that $\Ea(u) \geq \Ec(u)$, so that
\eqref{eq:defn_Lam} is satisfied and the dynamic atomistic and
continuum solutions will remain close for a macroscopic time interval
(cf. Theorem~\ref{th:w:mainthm}).

By contrast, in the case $a_i = a_i^{(2)}$, where $a_2 > 0$ we can see
from \eqref{eq:app_stab_Ea} that oscillations are energetically
advantageous. Indeed we note that, formally, defining
$\hat{\varphi}'(\xi) = (-1)^\xi$ gives infinite negative energy,
\begin{displaymath}
  \Ea(\hat{\varphi}) = \sum_{\xi\in\L} \B( -\smfrac{1}{2} |1|^2 +
  \smfrac{1}{4}  |0|^2 \B) = -\infty.
\end{displaymath}
Formally (since $\hat\varphi \notin \Ush^{1,2}$) one easily checks
that $H \hat\varphi=-\hat\varphi$, where $H := \ddel\Ea(0)$. A
straightforward approximation argument shows that $-1$ belongs to the
spectrum of $H$.

For the static case, this means that even if atomistic solutions
exist, they are not local minimizers.

For the dynamic case, it means that there exist exponentially growing
solutions. Using the characterization of the spectrum in terms of
approximate eigenfunctions, there exists for each $\delta > 0$ a
function $\psi_\delta \in \ell^2$ with $\| \psi_\delta \|_{\ell^2} =
1$, such that $\| H \psi_\delta + \psi_\delta \| \leq \delta$ (see
Section VIII.3 in \cite{ReedSimon}). Suppose now that we solve the
Cauchy--Born equation with $u(0) = \dot{u}(0) = 0$ and the atomistic
equation with $u(0) = 0$ and $\dot{u}(0) = \eps^2 \psi_\delta$. The
function $v(t) := \sinh(t) \eps^2 \psi_\delta$ then solves the
atomistic evolution equation to order $O(\delta)$. By estimating the
difference $u(t) - v(t)$ it is straightforward to prove that
\begin{displaymath}
  \| \dot{u}(t) \|_{\ell^2} \geq \eps^2 \smfrac12 e^t \qquad \text{for } t \leq 3 |\log\eps|,
\end{displaymath}
and in particular, $\|\dot{u} \|_{\ell^2}$ becomes of order one for $t
\sim |\log\eps|$.

\bibliographystyle{plain}
\bibliography{qc}
\end{document}

%% file: notation.tex
\renewcommand{\cases}[1]{\left\{ \begin{array}{rl} #1 \end{array} \right.}
\newcommand{\smfrac}[2]{{\textstyle \frac{#1}{#2}}}

\def\XXint#1#2#3{{\setbox0=\hbox{$#1{#2#3}{\int}$ }
\vcenter{\hbox{$#2#3$ }}\kern-.6\wd0}}

\def\b{\big}
\def\B{\Big}
\def\bg{\bigg}

\def\sep{\,|\,}
\def\bsep{\,\b|\,}

\def\conv{{\rm conv}}

\def\supp{{\rm supp}}

\def\R{\mathbb{R}}
\def\N{\mathbb{N}}
\def\Z{\mathbb{Z}}

\def\WW{W}
\def\CC{C}

\def\LL{L}

\def\WWh{\dot{W}}

\def\WWhz{\dot{W}_0}

\def\dx{\,{\rm d}x}
\def\dy{\,{\rm d}y}
\def\dz{\,{\rm d}z}

\def\dt{\,{\rm d}t}
\def\ds{\,{\rm d}s}
\def\dd{{\rm d}}
\def\pp{\partial}

\def\<{\langle}
\def\>{\rangle}

\def\ol{\overline}

\def\mA{{\sf A}}

\def\mF{{\sf F}}
\def\mG{{\sf G}}

\def\mS{{\sf S}}
\def\mO{{\sf 0}}

\def\bfg{{\bm g}}
\def\bfrho{{\bm \rho}}
\def\bfv{{\bm v}}
\def\bfh{{\bm h}}
\def\bfO{{\bm 0}}

\def\bbC{\mathbb{C}}

\newcommand{\Da}[1]{D_{\!#1}}
\newcommand{\Dc}[1]{\D_{#1}}
\def\D{\nabla}
\def\del{\delta}
\def\ddel{\delta^2}

\def\loc{{\rm loc}}

\def\eps{\varepsilon}

\def\a{{\rm a}}
\def\c{{\rm c}}

\def\L{\Lambda}
\def\Rg{\Lambda_*}

\def\vsig{\varsigma}

\def\Us{\mathscr{W}}
\def\Usz{\Us_0}
\def\Ush{\dot{\mathscr{W}}}

\def\E{E}
\def\Ea{\E^\a}
\def\Ec{\E^\c}
\def\Es{\Phi}

\def\Om{{\R^d}}

\def\tilu{\tilde u}
\def\tilv{\tilde v}

\def\tilw{\tilde w}

\def\barv{v}
\def\barw{w}
\def\baru{u}
\def\bare{e}
\def\barf{f}

\def\Sa{\mS^\a}
\def\Scb{\mS^\c}

\def\zz{\zeta}

\def\ww{\chi}

\def\DV{\mathscr{D}}

\def\Lam{\gamma}

\def\Ba{\mathscr{K}}

\def\Bc{K}

\def\inv{\kappa}

\def\Itil{I}